\documentclass[11pt]{article}

\usepackage{graphics}
\usepackage[dvips]{epsfig}
\usepackage{psfrag}
\usepackage{pinlabel}

\usepackage{amsmath}
\usepackage{amsfonts}
\usepackage{latexsym}
\usepackage{amssymb}
\usepackage[usenames]{color}
\usepackage{amsthm}

\input xy
\xyoption{all}

\theoremstyle{plain}
\newtheorem{defn}{Definition}[section]
\newtheorem{Thm}[defn]{Theorem}
\newtheorem{Prop}[defn]{Proposition}

\theoremstyle{remark}
\newtheorem{Rem}[defn]{Remark}

\DeclareMathAlphabet{\mathdj}{U}{msb}{m}{n}

\title{Bilinearised Legendrian contact homology and\\ the augmentation category}
\author{Fr\'ed\'eric Bourgeois \\ Baptiste Chantraine}

\date{}
\begin{document}

\maketitle
\begin{abstract}
  In this paper we construct an $\mathcal{A}_\infty$-category associated to a Legendrian submanifold of jet spaces. Objects of the category are augmentations of the Chekanov algebra $\mathcal{A}(\Lambda)$ and the homology of the morphism spaces forms a new set of invariants of Legendrian submanifolds called the bilinearised Legendrian contact homology. Those are constructed as a generalisation of linearised Legendrian contact homology using two augmentations instead of one. Considering similar constructions with more augmentations leads to the higher order compositions map in the category and generalises the idea of \cite{Productstructure} where an $\mathcal{A}_\infty$-algebra was constructed from one augmentation. This category allows us to define a notion of equivalence of augmentations when the coefficient ring is a field regardless of its characteristic. We use simple examples to show that bilinearised cohomology groups are efficient to distinguish those equivalences classes. We also generalise the duality exact sequence from \cite{Duality_EkholmetAl} in our context, and interpret geometrically the bilinearised homology in terms of the Floer homology of Lagrangian fillings (following \cite{Ekholm_FloerlagCOnt}).
\end{abstract}

\section{Introduction}
\label{sec:introduction}
For a Legendrian submanifold $\Lambda$ of a jet space $\mathcal{J}^1(M)$ Legendrian contact homology is the homology of a non-commutative differential algebra freely generated by double points of the projection of $\Lambda$ to $T^*(M)$ constructed in \cite{Chekanov_DGA_Legendrian} for $M\simeq \mathbb{R}$ and in \cite{Ekholm_Contact_Homology} and \cite{LCHgeneral} in general. We refer to this algebra as the \textit{Chekanov algebra} of $\Lambda$ denoted by $(\mathcal{A}(\Lambda), \partial)$ whose homology is denoted by $LCH(\Lambda)$. The boundary operator counts holomorphic curves in the symplectisation whose domains are disks with points removed on the boundary. At one of these points, the holomorphic curve is required to have one positive asymptotic and at all the others it is required to have a negative asymptotic. When non zero, Legendrian contact homology is most of time infinite dimensional, hence this turns out to be difficult to distinguish two non-isomorphic LCH algebras.

Linearisation of semi-free DGA's is a process used in \cite{Chekanov_DGA_Legendrian} which associates finite dimensional invariants of Legendrian submanifolds from their Chekanov algebras. Linearisation is done using augmentations of $\mathcal{A}(\Lambda)$ that is DGA-maps $\varepsilon:\mathcal{A}(\Lambda)\rightarrow\mathbb{K}$ where $\mathbb{K}$ is the coefficient field we consider (most frequentely $\mathbb{Z}_2$). Those invariants have the advantage to be both computable (once the boundary operator of Legendrian contact homology is known) and efficient to distinguish Legendrian submanifolds which were not distinguished by other invariants. However, the process of linearisation makes the theory commutative. Also, it is not efficient regarding the question of distinguishing augmentations of Chekanov algebra.

In this paper we propose a new invariant called \textit{bilinearised Legendrian contact homology} which closely follows the process of linearisation but uses the fact that the theory is non-commutative in order to use two augmentations instead of one. 

More precisely, for each pairs of augmentations $(\varepsilon_0,\varepsilon_1)$ we define a 
differential 
$$d^{\varepsilon_0,\varepsilon_1}: C(\Lambda)\rightarrow C(\Lambda)$$

in homology, and its dual

$$\mu^1_{\varepsilon_1,\varepsilon_0}: C^*(\Lambda)\rightarrow C^*(\Lambda)$$

in cohomology, where $C(\Lambda)$ is the vector space over $\mathbb{K}$ generated by Reeb chords of $\Lambda$. We denote by $C^{\varepsilon_0,\varepsilon_1}$ the complex $(C(\Lambda),d^{\varepsilon_0,\varepsilon_1})$, and by $C_{\varepsilon_1,\varepsilon_0}$ the complex $(C^*(\Lambda),\mu^1_{\varepsilon_1,\varepsilon_0})$. The homology of those complexes are the bilinearised (co)homology groups, we denote them by $LCH^{\varepsilon_0,\varepsilon_1}(\Lambda)$ and $LCH_{\varepsilon_1,\varepsilon_0}(\Lambda)$ respectively. Those are generalisations of the standard linearised Legendrian contact homology, as it will appear obvious from the definition that the differential $d^{\varepsilon,\varepsilon}$ is the standard augmented differential $d^\varepsilon$ from \cite{Chekanov_DGA_Legendrian} and \cite{Ekholm_Contact_Homology}.

\begin{figure}[ht!]
\labellist
\small\hair 2pt
\pinlabel {$\gamma^+$} [bl] at 160 260
\pinlabel {$\varepsilon_0(\gamma_1^-)$} [tl] at 30 25
\pinlabel {$\varepsilon_0(\gamma_2^-)$} [tl] at 115 25
\pinlabel {$\gamma^-_3$} [tl] at 205 25
\pinlabel {$\varepsilon_1(\gamma^-_4)$} [tl] at 290 25
\endlabellist
  \centering
  \includegraphics[height=7cm]{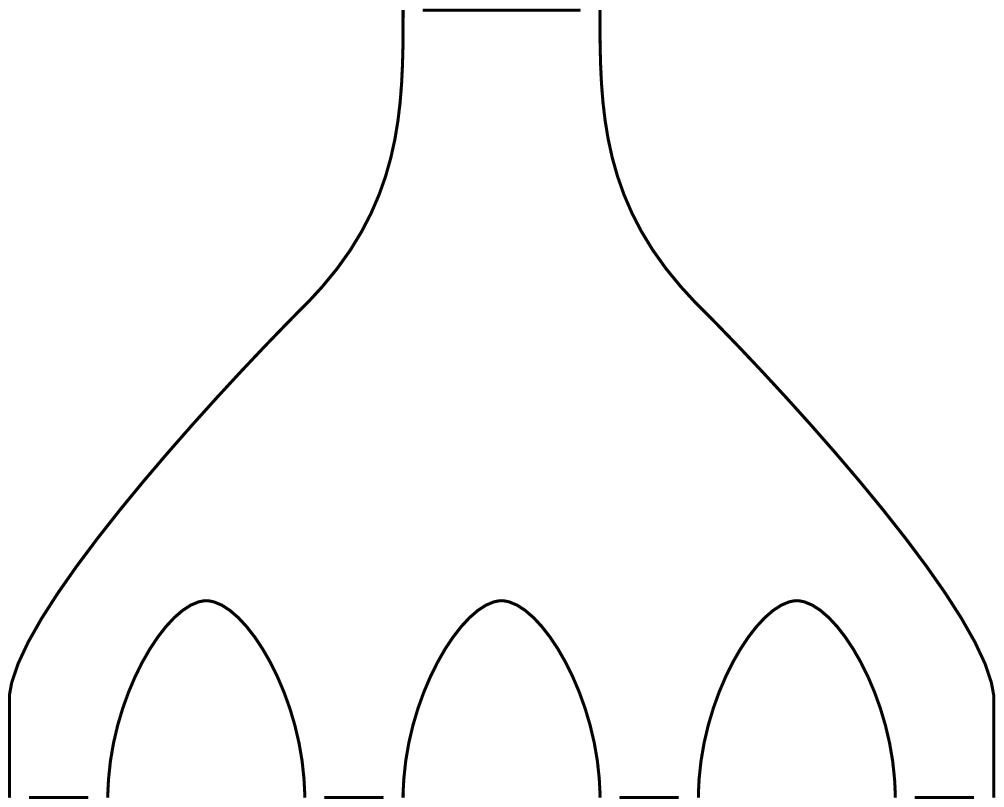}
  \caption{Curve contributing to $d^{\varepsilon_0;\varepsilon_1}(\gamma^+)$ (or  $\mu^1_{\varepsilon_0,\varepsilon_1}(\gamma_3^-)$).}
  \label{fig:diff}
\end{figure}

Our first result is that the set of those homologies is a Legendrian isotopy invariant as stated in the following

\begin{Thm}\label{thm:legisotinv}
  The set of isomorphism classes of $LCH_{\varepsilon_1,\varepsilon_0}(\Lambda)$ (or $LCH^{\varepsilon_0,\varepsilon_1}(\Lambda)$) over all pairs of augmentations $(\varepsilon_0,\varepsilon_1)$ of $\mathcal{A}(\Lambda)$ is a Legendrian isotopy invariant. 
\end{Thm}

Geometrically, bilinearised Legendrian contact homology associated to two augmentations $\varepsilon_0$ and $\varepsilon_1$ amounts to counting the holomorphic curves decorated with the augmentations. An example of such a decorated curves is shown in Figure \ref{fig:diff} note that the number of negative asymptotics on the left and on the right could be arbitrary (see Equation \eqref{eq:1} for an explicit algebraic definition).

We can use similar ideas and consider holomorphic curves decorated by $d+1 \ge 3$ augmentations to build structural maps $\mu^d_{\varepsilon_d,\ldots ,\varepsilon_0} : C_{\varepsilon_d,\varepsilon_{d-1}} \otimes \ldots \otimes C_{\varepsilon_1,\varepsilon_0}\rightarrow C_{\varepsilon_d,\varepsilon_0}$. Figure \ref{fig:comp} gives an example of such a curves, again the number of decorated negative ends is in general arbitrary, Equation \eqref{eq:9} gives the explicit formula of this operation.

Those structural maps satisfy an $\mathcal{A}_\infty$-relation which allows us to define an $\mathcal{A}_\infty$-category whose objects are augmentations of the Chekanov algebra of $\Lambda$. We call this category the \textit{augmentation category} of $\Lambda$. We denote it by $\mathit{Aug}(\Lambda)$. This is a direct generalisation of the $\mathcal{A}_\infty$-algebra constructed in \cite{Productstructure}. Under Legendrian isotopy this category changes by a pseudo-equivalence as stated in the following

\begin{Thm}
  \label{thm:quasiequilegisot}
If $\Lambda$ is Legendrian isotopic to $\Lambda'$ then the categories $\mathit{Aug}(\Lambda)$ and $\mathit{Aug}(\Lambda')$ are pseudo-equivalent.
\end{Thm}

\begin{figure}[ht!]
\labellist
\small\hair 2pt
\pinlabel {$\gamma^+$} [bl] at 165 260
\pinlabel {$\varepsilon_0(\gamma_1^-)$} [tl] at 1 25
\pinlabel {$\gamma^-_2$} [tl] at 85 25
\pinlabel {$\varepsilon_1(\gamma^-_3)$} [tl] at 175 25
\pinlabel {$\gamma_4^-$} [tl] at 260 25
\pinlabel {$\varepsilon_2(\gamma_5^-)$} [tl] at 345 25
\endlabellist
  \centering
  \includegraphics[height=7cm]{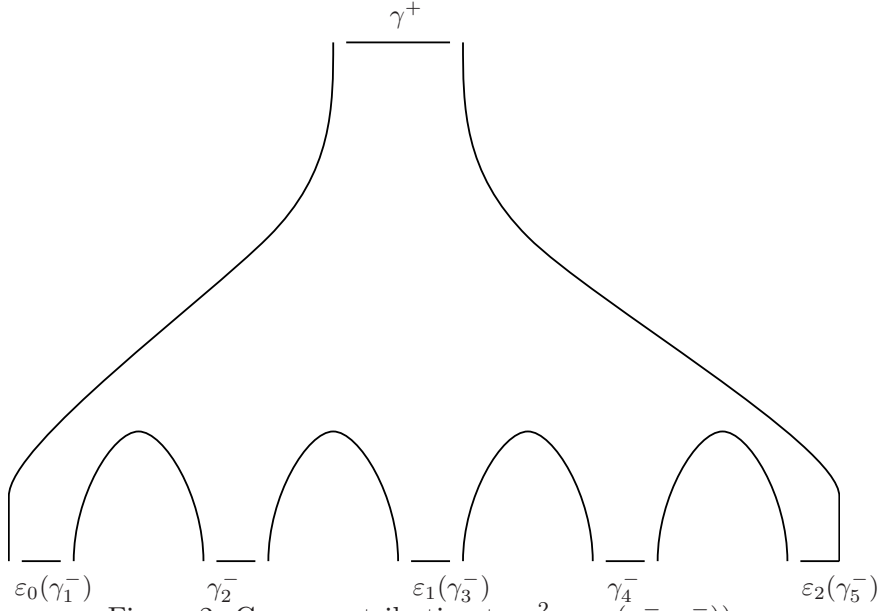}
  \caption{Curve contributing to $\mu^2_{\varepsilon_0,\varepsilon_1,\varepsilon_2}(\gamma_2^-,\gamma^-_4)$).}
  \label{fig:comp}
\end{figure}

These categories are not (in general) cohomologically unital (in the sense of \cite{Seidel_Fukaya}). Therefore one cannot take the definition of quasi-equivalence as in \cite{Seidel_Fukaya}. In Section \ref{sec:equiv-categ} we introduce the notion of pseudo-equivalence which will coincide with the notion of quasi-equivalence in the case of $c$-unital categories. We borrow many notations from \cite{Seidel_Fukaya} and try to give the necessary definitions to understand the main concepts in the present paper. As it will follow from this section that the morphism spaces in the corresponding homological categories are invariant under quasi-equivalences, Theorem \ref{thm:legisotinv} is a corollary of Theorem \ref{thm:quasiequilegisot}.

The benefit of having a category where the objects are augmentations of $\mathcal{A}(\Lambda)$ is that it allows to define the notion of equivalence of augmentations even when the ground field $\mathbb{K}$ is not of characteristic $0$ (compare with the definition in \cite{Bourgeois_Survey}). After precisely defining equivalence of augmentations in Section \ref{sec:equiv-augm} we prove that

\begin{Thm}\label{thm:equiclassisot}
 If $\Lambda$ is Legendrian isotopic to $\Lambda'$ then the quasi-equivalence of Theorem \ref{thm:quasiequilegisot} induces a bijection between the equivalence classes of augmentation of $\mathcal{A}(\Lambda)$ and those of $\mathcal{A}(\Lambda')$.
\end{Thm}

In particular if the characteristic of $\mathbb{K}$ is finite one gets that the (finite) number of equivalence classes is a Legendrian invariant.

The group $LCH^{\varepsilon_0,\varepsilon_1}$ appears to be an efficient tool to distinguish some of those equivalence classes of augmentations using the following theorem.

\begin{Thm}\label{thm:equivaugisom}
  If $\varepsilon_1$ and $\varepsilon_2$ are equivalent then for all augmentation $\varepsilon$
$$LCH^{\varepsilon_1,\varepsilon}(\Lambda)\simeq LCH^{\varepsilon_2,\varepsilon}(\Lambda)$$
and
$$LCH^{\varepsilon,\varepsilon_1}(\Lambda)\simeq LCH^{\varepsilon,\varepsilon_2}(\Lambda).$$
\end{Thm}

The first part of the paper (Section \ref{sec:algebraic-setup}) is devoted to the algebraic construction of the augmentation category for any semi-free differential graded algebra, we borrow the necessary definitions regarding $\mathcal{A}_{\infty}$-categories from \cite{Seidel_Fukaya} and try to make this paper as self-contained as possible. In the second part, we apply this algebraic construction to the case of the Chekanov algebra of a Legendrian submanifold. We also give a geometrical interpretation of the bilinearised differential in terms of the Chekanov algebra of the $2$-copy Legendrian link. We also investigate in Section \ref{sec:prod-struct-adn} the generalisation of the duality exact sequence from \cite{Duality_EkholmetAl}. Where we prove the following theorem.
\begin{Thm}\label{thm:duality}
  Let $\varepsilon_0$ and $\varepsilon_1$ be two augmentations of $\Lambda$ where $\Lambda\subset \mathcal{J}^1(M)$ is horizontally displaceable. Then there is a long exact sequence
  \begin{equation}
    \label{eq:17}
    \cdots\rightarrow H_{k+1}(\Lambda)\stackrel{\sigma^{\varepsilon_0,\varepsilon_1}}{\longrightarrow} LCH^{n-k-1}_{\varepsilon_1,\varepsilon_0}(\Lambda)\rightarrow LCH_{k}^{\varepsilon_0,\varepsilon_1}(\Lambda)\stackrel{\rho^{\varepsilon_0,\varepsilon_1}}{\longrightarrow} H_k(\Lambda)\rightarrow\cdots
  \end{equation}
\end{Thm}

In Section \ref{sec:geom-interpr-footn} we interpret the bilinearised Legendrian contact homology when $\varepsilon_1$ and $\varepsilon_2$ come from Lagrangian fillings in terms of the Lagrangian Floer homology of those fillings (as defined in \cite{Ekholm_FloerlagCOnt}). In Section \ref{sec:exemple} we provide some simple example of computation which demonstrates the effectiveness of Theorem \ref{thm:equivaugisom} to distinguish augmentations.

{\bf Acknowledgements.} FB was partially supported by ERC Starting Grant StG-239781-ContactMath.
BC was partially supported by the Fonds de la Recherche Scientifique (FRS-FNRS), Belgium. 
Both of us wish to thank Petya Pushkar for many inspiring conversations without which this project might never have been initiated. The present work is part of the authors activities within CAST, a Research Network 
Program of the European Science Foundation.

\tableofcontents

\section{Algebraic setup}
\label{sec:algebraic-setup}

\subsection{The $n$-copy  algebra of a free DGA}
\label{sec:n-copy-algebra}

\paragraph{The $n$-copy algebra.}
\label{sec:module}

Let $\mathbb{K}$ be a field. For each $n\in\mathbb{N}$ we denote by $\mathbb{K}_n$ the semi-simple algebra over $\mathbb{K}$ generated by $e_i$, $i\in\{1,\ldots, n\}$, with the relations $e_i\cdot e_j=\delta_{i,j} e_i$ and $\sum e_i=1$. For a set $A$ we denote by $C_n(A)$ the free $\mathbb{K}_n$-bimodule generated by the elements of $A$ and by $$\mathcal{A}_n(A)=\mathbb{K}_n\oplus C_n(A)\oplus (C_n(A)\otimes C_n(A))\oplus\ldots \oplus C_n(A)^{\otimes k}\oplus\ldots$$ the tensor algebra over $\mathbb{K}_n$ of $C_n(A)$ (here $\otimes$ denotes the tensor product of $\mathbb{K}_n$-modules). For each $a\in A$ we denote by $a_{i,j}$ the element $e_i\cdot a\cdot e_j\in C_n(A)$. Note that as a $\mathbb{K}_n$-bimodule $C_n(A)$ decomposes as $\bigoplus_{(i,j)}C_n(A)_{i,j}$ where $C_n(A)_{i,j}$ is the submodule generated by the $a_{i,j}$. 
In the tensor algebra $\mathcal{A}_n(A)$, $C_n(A)_{i,j}\otimes C_n(A)_{k,l}$ is non zero iff $j=k$. 
A pair of multi-indices $(I,J)$ of the same length  $k$, i.e. $I = (i_1, \ldots, i_k)$
and $J=(j_1, \ldots, j_k)$, is said to be \textit{composable} if for every $l=1,\ldots, k-1$ we have 
$j_l=i_{l+1}$. Then the tensor product $C_n(A)_{I,J} := C_n(A)_{i_1,j_1} \otimes \ldots \otimes C_n(A)_{i_k,j_k}$ does not vanish iff $(I,J)$ is composable, so that the tensor algebra 
$\mathcal{A}_n(A)$ decomposes as a $\mathbb{K}_n$-bimodule as the direct sum of submodules $C_n(A)_{I,J}$ over all composable pairs $(I,J)$. Note that, for $i= 1, \ldots, n$, the 
$\mathbb{K}$-subalgebra $\mathcal{A}_n(A)_i \subset \mathcal{A}_n(A)$  
defined as the tensor algebra of $C_n(A)_{i,i}$, and corresponding to multi-indices $I$ and $J$ 
having all components equal to $i$, is naturally isomorphic to 
$\mathcal{A}(A):=\mathcal{A}_1(A)$ as a $\mathbb{K}$-algebra.

Now suppose that the tensor algebra $\mathcal{A}(A)$ is equipped with a differential $\partial$ such that $(\mathcal{A}(A),\partial)$ is a differential algebra. We define a differential $\partial_n$ on 
$\mathcal{A}_n(A)$ by setting $\partial_n(e_i \cdot a \cdot e_j)= e_i \cdot \partial(a) \cdot e_j$
for all $a \in A$, where $\partial(a)$ is interpreted as an element of $\mathcal{A}_n(A)$.
We then extend $\partial_n$ to $\mathcal{A}_n(A)$ by linearity and by the Leibniz rule.
Note that since $a=\sum_{i,j} a_{i,j}$ in $\mathcal{A}_n(A)$ and $\sum_i e_i = 1$ in $\mathbb{K}_n$
the formula tells us that any word of length $k$ in $\partial(a)$ appears as a sum in $\partial_n(a)$ 
over all composable pairs of multi-indices of length $k$.

\paragraph{Grading.}
\label{sec:grading}

Assume also that the set $A$ comes with a grading map $\text{gr}: A\rightarrow \mathbb{Z}$ which defines a grading on homogeneous elements of $\mathcal{A}(A)$ by $\text{gr}(ab)=\text{gr}(a) +\text{gr}(b)$. Defining the grading of elements of $\mathbb{K}_n$ to be $0$, this extends to a grading on $\mathcal{A}_n(A)$. From now on we assume that the differential $\partial$ has degree $-1$, so
that $\partial_n$ has degree $-1$ as well and turns $(\mathcal{A}_n(A),\partial_n)$ a differential graded algebra (DGA) over $\mathbb{K}_n$. This means that the Leibniz rule becomes: $\partial(ab)=\partial(a)\cdot b+(-1)^{gr(a)}a\cdot\partial(b)$. The DGA $(\mathcal{A},\partial)$ with generating set $A$ is commonly referred in the litterature to a \textit{semi-free} DGA.

In order to agree with the degree convention in the literature (see \cite{Seidel_Fukaya}), there is a stabilisation process when going from the language of differential graded algebras to the language of $\mathcal{A}_\infty$-categories which induces a difference between the notions of grading. We will give two different notations for those two different gradings. We denote by $gr(a)$ the grading of $a\in A$ as a generator of $\mathcal{A}$, this grading will be referred to as the \textit{reduced} grading according to the literature. We will denote by $\vert a\vert=gr(a)+1$ the \textit{absolute} grading of $a$ as a generator of the dual module $C^*(A)$. Although we will pay careful attention to our notations, the convention is that whenever we speak in the world of $\mathcal{A}_\infty$-categories the grading is assumed to be the absolute grading and in the world of DGA the grading is assumed to be the reduced grading. 

If $V$ is a graded vector space, we denote by $V[d]$ the graded vector space isomorphic to $V$ with degree shifted by $d$ (i.e. the grading of $a$ in $V[d]$ is $gr(a)-d$). When defined between graded vector space, we assume that all our linear maps preserve the degree hence a map $F: V\rightarrow W$ shifting the degree by $d$ will be denoted by $F:V\rightarrow W[d]$. 

If $T: (V,\mu^1_V)\rightarrow (W[d],\mu^1_W)$ is a linear map between chain complexes, we will sometimes denote for short $T\circ \mu^1_V + (-1)^d\mu^1_W\circ T$ by $\mu^1(T)$. 

\paragraph{Augmentations of $\mathcal{A}_n(A)$.}
\label{sec:augm-mathc}

Recall that an augmentation of a DGA $\mathcal{A}$ over $\mathbb{K}$ 
is a DGA map from $(\mathcal{A},\partial)$ to $(\mathbb{K},0)$. More precisely it as a map $\varepsilon:\mathcal{A}\rightarrow\mathbb{K}$ satisfying
\begin{itemize}
\item $\varepsilon(a)=0$ if $gr(a)\not=0$.
\item $\varepsilon\circ\partial =0$.
\item $\varepsilon(ab)=\varepsilon(a)\cdot\varepsilon(b)$.
\end{itemize}
Let $E = (\varepsilon_1, \ldots, \varepsilon_n)$ be a $n$-tuple of augmentations of $\mathcal{A}(A)$ 
over $\mathbb{K}$. We define $\varepsilon_E \colon \mathcal{A}_n(A)\rightarrow \mathbb{K}_n$ by 
$\varepsilon_E(a)=\sum_i \varepsilon_i(e_i\cdot a\cdot e_i) \cdot e_i$. In other words,  
$\varepsilon_E$ equals $\varepsilon_i \cdot e_i$ on $\mathcal{A}(A)_i$, and $\varepsilon_E$ vanishes
on $C_n(A)_{I,J}$ when the multi-indices $I$ and $J$ are not constant.

\begin{Prop}\label{sec:augm-mathc-prop}
  $\varepsilon_E$ is an augmentation of $\mathcal{A}_n(A)$ over $\mathbb{K}_n$.
\end{Prop}

\begin{proof}

First, one must check that $\varepsilon_E$ is an algebra morphism. For $a \in C_n(A)_{I,J}$
and $b \in C_n(A)_{K,L}$, we have $\varepsilon_E(a \cdot b) = \varepsilon_i(a \cdot b) \cdot e_i = \varepsilon_i(a) \varepsilon_i(b) \cdot e_i = \varepsilon_E(a) \varepsilon_E(b)$ if $I, J, K, L$ have
all components equal to $i$. Otherwise, $a \cdot b = 0$ so that  $\varepsilon_E(a \cdot b) = 0$
and at least one of $\varepsilon_E(a)$ or $\varepsilon_E(b)$ vanishes as well. 

Second,  one must check that $\varepsilon_E \circ\partial_n=0$. Since $\partial_n(e_i \cdot a \cdot e_j) = e_i\cdot\partial(a)\cdot e_j$, we have $\varepsilon_E\circ\partial_n(e_i \cdot a \cdot e_j)=
e_i \cdot \varepsilon_i\circ\partial(a) \delta_{i,j}$ which vanishes since $\varepsilon_i$ is an augmentation.
\end{proof}

Using the augmentation $\varepsilon_E$ one defines the $\mathbb{K}_n$-algebra map $\phi_{\varepsilon_E}:\mathcal{A}_n\rightarrow\mathcal{A}_n$ which sends a generator $a$ to $a+\varepsilon_E(a)$. Consider the induced differential $\partial^{\varepsilon_E}=\phi_{\varepsilon_E}\circ\partial_n\circ\phi_{\varepsilon_E}^{-1}$. Since $\varepsilon_E$ is an augmentation, $\partial^{\varepsilon_E}\vert_{C_n(A)}$ has no constant term and hence decomposes as a sum $\partial_1^{\varepsilon_E}\oplus\partial_2^{\varepsilon_E}\oplus\ldots\oplus\partial_d^{\varepsilon_E}\oplus\ldots$ where $\partial_d^{\varepsilon_E}:C_n(A)\rightarrow C_n(A)^{\otimes d}[-1]$. From $\partial^{\varepsilon_E}\circ\partial^{\varepsilon_E}=0$ one gets that $\partial_1^{\varepsilon_E}\circ\partial_1^{\varepsilon_E}=0$.

Dualising each of $\partial_d^{\varepsilon_E}$ one gets a family of maps $\mu^{d,n}_{\varepsilon_E}: C^*_n(A)^{\otimes d}\rightarrow C^*_n(A)[2-d]$ (using the absolute grading here) define by the relation:
$$\mu^{d,n}(f)(b)=f(\partial_d^{\varepsilon_E}(b))$$

for any $\mathbb{K}_n$-balanced forms $f$ on $C_n(A)$.

Similarly to \cite{Productstructure} from $\partial^{\varepsilon_E}\circ\partial^{\varepsilon_E}=0$ one gets that the family of maps $\mu^{d,n}_{\varepsilon_E}$ is an $\mathcal{A}_n$-algebra structure on $C^*_n(A)$ that is:
\begin{equation}
\sum_{i=1}^d \sum_{j=0}^{d-i}(-1)^{\maltese_j}\mu_{\varepsilon_E}^{d-i+1,n}\big(a_d,a_{d-1},\ldots,a_{j+i+1}, \mu_{\varepsilon_E}^{i,n}(a_{j+i},\ldots, a_{j+1}),a_j,\ldots, a_1\big)=0\label{eq:11}
\end{equation}

where $\maltese_j=\vert a_1\vert+\ldots + \vert a_j \vert-j$.

\subsection{The augmentation category}
\label{sec:augm-categ}

\paragraph{Bilinearised complexes.}
\label{sec:augm-categ-1}

Note that $\mu^{d,n}_{\varepsilon_E}$ maps $C^*_{i_d,i_{d-1}}\otimes C^*_{i_{d-1},i_{d-2}}\otimes\cdots\otimes C^*_{i_1,i_0}$ to $C^*_{i_d,i_0}[2-d]$ and that this maps depends only on $\partial$ and $\varepsilon_{i_d}, \varepsilon_{i_{d-1}},\ldots, \varepsilon_{i_0}$. This implies that, after identifying each $C^*_{i,j}$ with $C^*(A)$ for each $(d+1)$-tuple of augmentations $\varepsilon_0,\ldots, \varepsilon_d$ the map $\mu^{d,n}_{\varepsilon_E}$ induces a map $\mu^d_{\varepsilon_d,\cdots,\varepsilon_0}:C^*(A)^{\otimes d}\rightarrow C^*(A)[2-d]$ independent of $n\geq d+1$ and the remaining augmentations in $E$. A similar discussion holds for the maps $\partial^{\varepsilon_E}_d: C(A)\rightarrow C(A)[-1]$. From equation \eqref{eq:11} one deduces that:

\begin{equation}
\sum_{i=1}^d \sum_{j=0}^{d-i}(-1)^{\maltese_j}\mu_{\varepsilon_d,\cdots,\varepsilon_{j+i},\varepsilon_{j},\ldots, \varepsilon_0}^{d-i+1}\big(a_d,a_{d-1},\ldots,a_{j+i+1},\mu_{\varepsilon_{j+i},\ldots, \varepsilon_{j}}^i(a_{j+i},\ldots, a_{j+1}),a_j,\ldots,a_1\big)=0.\label{eq:12}
\end{equation}

In particular when $d=1$ one gets $\mu^1_{\varepsilon_1,\varepsilon_0}\circ \mu^1_{\varepsilon_1,\varepsilon_0}=0$, similarly defining $d^{\varepsilon_0,\varepsilon_1}$ the map on $C(A)$ associated to $\partial_1^{\varepsilon_E}$ one gets $d^{\varepsilon_0,\varepsilon_1}\circ d^{\varepsilon_0,\varepsilon_1}=0$. Hence we get two complexes called the \textit{bilinearised complexes}, $C_{\varepsilon_1,\varepsilon_0}=(C^*(A),\mu^1_{\varepsilon_1,\varepsilon_0})$ and $C^{\varepsilon_0,\varepsilon_1}=(C(A),d^{\varepsilon_0,\varepsilon_1})$. The homology of $C^{\varepsilon_0,\varepsilon_1}$ is called the \textit{bilinearised homology} of $(\mathcal{A}(A),\partial)$ with respect to $\varepsilon_0$ and $\varepsilon_1$ and the homology of $C_{\varepsilon_1,\varepsilon_0}$ is called the \textit{bilinearised cohomology}.

The maps $\mu^d_{\varepsilon_d,\ldots,\varepsilon_0}$ are higher order compositions of an $\mathcal{A}_\infty$-category we describe here.

\paragraph{Definition of $\textit{Aug(A)}$.}
\label{sec:defin-text}

The augmentation category $\textit{Aug(A)}$ is the $\mathcal{A}_\infty$-category defined by
\begin{itemize}
\item $Ob(\mathit{Aug(A)})$ is the set of augmentations of
  $\mathcal{A}$.  
\item Morphisms from $\varepsilon_1$ to $\varepsilon_2$ are the complexes
  $C_{\varepsilon_2,\varepsilon_1}=(C^*(A),\mu^1_{\varepsilon_2,\varepsilon_1}$).  
\item The $\mathcal{A}_\infty$-composition maps are
$$
\mu^d_{\varepsilon_d,\ldots,\varepsilon_0}:C_{\varepsilon_d,\varepsilon_{d-1}}\otimes\ldots\otimes C_{\varepsilon_1,\varepsilon_0}\rightarrow C_{\varepsilon_d,\varepsilon_0}[2-d].
$$
\end{itemize}

In order to facilitate computations, we unravel here the construction to give an explicit formula for those maps. If $\partial(a_j)=\sum\limits_{i_1, \ldots, i_k} x^j_{i_1, \ldots, i_k} \cdot b^j_{i_1}\ldots b^j_{i_k}$ 
with $x^j_{i_1, \ldots, i_k} \in \mathbb{K}$ and $b^j_{i_1}, \ldots, b^j_{i_k} \in A$ then

\begin{align}
  \label{eq:13}
&\partial^{\varepsilon_0,\ldots,\varepsilon_d}_d(a_j)=
\sum\limits_{i_1, \ldots, i_k}\sum_{j_1<j_2<\ldots<j_d}x^j_{i_1, \ldots, i_k}\varepsilon_0(b^j_{i_1})\ldots \varepsilon_0(b^j_{i_{j_1-1}})\varepsilon_1(b^j_{i_{j_1+1}})\ldots\varepsilon_{d-1}(b^j_{i_{j_{d}-1}})
\nonumber\\
&\hspace{6cm} \varepsilon_d(b^j_{i_{j_d+1}})\ldots\varepsilon_d(b^j_{i_k})\cdot b^j_{i_{j_1}}b^j_{i_{j_2}}\ldots b^j_{i_{j_d}}  
\end{align}

and 

\begin{align}
  &\mu^d_{\varepsilon_d,\ldots,\varepsilon_0}(b_d,\ldots,
  b_1)=\nonumber\\
&\sum\limits_{a_j\in A}
   \sum_{l_d < \ldots < l_1}
    \sum\limits_{\stackrel{i_1,\ldots,i_k}{b_d=b^j_{i_{l_d}},\ldots,
      b_1=b^j_{i_{l_1}}}}x^j_{i_1, \ldots, i_k}\varepsilon_d(b^j_{i_1})\ldots\varepsilon_d(b^j_{i_{l_d-1}})\varepsilon_{d-1}(b^j_{i_{l_{d}+1}})\ldots\varepsilon_0(b^j_{i_{l_1}+1})\ldots\varepsilon_0(b^j_{i_k})\cdot
  a_j\label{eq:9}
  \end{align}
where we do a slight abuse and identify $a_j$ with its dual generator.

To make things clearer, if $w_db_dw_{d-1}b_{d-1}w_{d-2}\ldots w_1b_1w_0$ is a word in $\partial(a)$ then we have a corresponding contribution 
$\varepsilon_d(w_d)\ldots\varepsilon_1(w_1)\varepsilon_0(w_0)\cdot a$ in $\mu^{d}_{\varepsilon_d,\ldots,\varepsilon_0}(b_d, \ldots, b_1)$ (see Figures \ref{fig:diff} and \ref{fig:comp}). When obvious we will drop the subscripts and write $\mu^d$ instead of $\mu^d_{\varepsilon_d,\ldots,\varepsilon_0}$.

The most useful special cases of Equations \eqref{eq:13} and \eqref{eq:9} are when $d=1$ which gives
\begin{align}
  \label{eq:1}
  d^{\varepsilon_0,\varepsilon_1}(a_j)&=\sum\limits_{i_1, \ldots, i_k}\sum\limits_{l=1\ldots k}x^j_{i_1, \ldots, i_k}\varepsilon_0(b_{i_1}^j)\ldots \varepsilon_0(b^j_{i_l-1})\varepsilon_1(b^j_{i_l+1})\ldots\varepsilon_1(b^j_{i_k})\cdot b^j_{i_l}\\
\mu^1_{\varepsilon_1,\varepsilon_0}(b)&=\sum_{a_j\in A}\sum\limits_{\stackrel{i_1, \ldots, i_k}{b=b^j_{i_l}}}x^j_{i_1, \ldots, i_k}\varepsilon_1(b_{i_1}^j)\ldots \varepsilon_1(b^j_{i_l-1})\varepsilon_0(b^j_{i_l+1})\ldots\varepsilon_0(b^j_{i_k})\cdot a_j.
\end{align}

The homological category of $\mathit{Aug}(A)$ is $H(\mathit{Aug}(A))$ whose objects are the same as those of $\mathit{Aug}(A)$, the morphism spaces are $Hom(\varepsilon_0,\varepsilon_1)=H(C_{\varepsilon_1,\varepsilon_0})$ and the composition of morphism is given by $[b_1]\circ [b_2]=(-1)^{\vert b_2 \vert}[\mu_{\varepsilon_2,\varepsilon_1,\varepsilon_0}^2(b_2,b_1)]$ (note the case $d=2$ of Equation \eqref{eq:12} implies that this is well defined). The case $d=3$ of Equation \eqref{eq:9} implies that this composition is associative, hence (except for the existence of an identity morphism) $H(\mathit{Aug}(A))$ is a genuine category.
 
\subsection{Equivalence of non-unital $\mathcal{A}_\infty$-categories}
\label{sec:equiv-categ}

By the stabilisation process for the grading when going from the world of DGA to the one of $\mathcal{A}_\infty$-category, generators of degree $0$ in the morphism spaces are elements $a$ of $A$ such that $gr(a)=-1$. This implies that, by degree considerations, the augmentation category has no reason for being unital (in any of the sense of \cite{Seidel_Fukaya}). The notion of equivalence of categories in this context is not immediate. For convenience for the reader we will recall here the basic definitions of $\mathcal{A}_\infty$-functors and natural transformations from \cite{Seidel_Fukaya}. Then we will propose a definition for equivalence of $\mathcal{A}_\infty$-categories which do not necessarily have units (cohomological or strict). For objects $\varepsilon_0$ and $\varepsilon_1$ in an $\mathcal{A}_\infty$-category $\mathcal{A}$ we maintain the notation $C_{\varepsilon_1,\varepsilon_0}$ for the morphism spaces from
$\varepsilon_0$ to $\varepsilon_1$. As we will consider several categories, we will sometimes denote those morphism spaces $C^{\mathcal{A}}_{\varepsilon_1,\varepsilon_0}$ in order to specify in which category the morphism space is considered.

\begin{defn}
  An $\mathcal{A}_\infty$-functor $\mathcal{F}$ between two $\mathcal{A}_\infty$-categories $(\mathcal{A},\{\mu^d_{\mathcal{A}}\})$ and $(\mathcal{B},\{\mu^d_{\mathcal{B}}\})$ consists of the following:
  \begin{itemize}
  \item {A map $F: Ob(\mathcal{A})\rightarrow Ob(\mathcal{B})$}
\item {For each $d\geq 1$ and $(\varepsilon_0,\varepsilon_1,\ldots, \varepsilon_d)\in Ob(\mathcal{A})^{d+1}$, a map $F^d_{\varepsilon_d,\ldots, \varepsilon_0}: C_{\varepsilon_d,\varepsilon_{d-1}}\otimes\ldots\otimes C_{\varepsilon_1,\varepsilon_0}\rightarrow C_{F(\varepsilon_d),F(\varepsilon_0)}[1-d]$ satisfying:
    \begin{align}\label{eq:4}
      &\sum_{r=1}^d\sum_{s_1+\ldots + s_r=d} \mu_{\mathcal{B}}^r\big( F^{s_r}_{\varepsilon_d,\ldots, \varepsilon_{d-s_r}}(a_d,\ldots,a_{d-s_r+1}),\ldots ,F^{s_1}_{\varepsilon_{s_1},\ldots, \varepsilon_0}(a_{s_1},\ldots, a_1)\big)\\
      =&\sum_{i=1}^d\sum_{j=0}^{d-i}(-1)^{\maltese_j} F^{d-i+1}_{\varepsilon_d,\ldots, \varepsilon_{j+i},\varepsilon_j,\ldots, \varepsilon_0}\big(a_d,\ldots, \mu_{\mathcal{A}}^i(a_{j+i},\ldots ,a_{j+1}),\ldots , a_1\big).\nonumber
    \end{align}
    }
  \end{itemize}
\end{defn}

The terms $d=1,2$ of equation \eqref{eq:4} imply that $F^1_{\varepsilon_1,\varepsilon_0}$ descends to a map $H(F^1): H(C_{\varepsilon_1,\varepsilon_0})\rightarrow H(C_{F(\varepsilon_1),F(\varepsilon_0)})$ which is a functor from $H(\mathcal{A})$ to $H(\mathcal{B})$. Again we drop the subscripts from the notation as they are most of the time obvious and write $F^d$ for $F^d_{\varepsilon_d,\ldots, \varepsilon_0}$.

We now recall the definition of a pre-natural transformation between $\mathcal{A}_\infty$-functors. 

\begin{defn}
  Let $\mathcal{F}$ and $\mathcal{G}$ be two $\mathcal{A}_\infty$-functors from $\mathcal{A}$ to $\mathcal{B}$. A pre-natural transformation of degree $g$ from $\mathcal{F}$ to $\mathcal{G}$ is a family $T = \{T^d\}_{d \ge0}$ where each $T^d$ consists of maps $T^d_{\varepsilon_d,\ldots, \varepsilon_0}: C_{\varepsilon_d,\varepsilon_{d-1}}\otimes\ldots\otimes C_{\varepsilon_1,\varepsilon_0}\rightarrow C_{G(\varepsilon_d),F(\varepsilon_0)}[g-d]$.
\end{defn}

The set of pre-natural transformations comes with a degree $1$ differential defined by

  \begin{align}\label{eq:3}
    \mu^1(T)^d=&\sum_{r=1}^{d+1}\sum_{i=1}^r\sum_{s_1+\cdots +s_r=d} (-1)^\dagger \mu_{\mathcal{B}}^r\big(G^{s_r}(a_d,\ldots , a_{d-s_r+1}), \ldots,
    G^{s_{i+1}}(a_{s_1+\ldots+s_{i+1}},\ldots, a_{s_1+\ldots+s_{i}+1}), \nonumber\\
& \hspace{2cm} T^{s_i}_{\varepsilon_{s_1+\ldots+s_{i}},\ldots,\varepsilon_{s_1+\ldots+s_{i-1}}}(a_{s_1+\ldots + s_i},\ldots ,a_{s_1+\ldots +s_{i-1}+1}),\\
&\hspace{2cm} F^{s_{i-1}}(a_{s_1+\ldots +s_{i-1}},\ldots,
    a_{{s_1+\ldots +s_{i-2}+1}}), \ldots, F^{s_1}(a_{s_1},\ldots,
    a_{1})\big) \nonumber\\
&-\sum_{i=1}^d\sum_{j=0}^{d-i} (-1)^{\maltese_j+g-1} T^{d-i+1}_{\varepsilon_{d},\ldots ,\varepsilon_{j+i},\varepsilon_{j},\ldots, \varepsilon_0}\big(a_d,\ldots, a_{j+i+1},\mu_{\mathcal{A}}^i(a_{j+i},\ldots,
    a_{j+1}),a_j,\ldots, a_1\big)\nonumber
      \end{align}
where $\dagger=(g-1)(\vert a_1\vert+\ldots +\vert a_{s_1+\ldots + s_{i-1}}\vert -s_1-\ldots -s_{i-1})$.

A \textit{natural transformation} is a pre-natural transformation such that $\mu^1(T)=0$.
For convenience of the reader we will explicitly detail the case $d=0,1$ and $2$ of $\mu^1(T)^d=0$ as it will be useful later.

\begin{itemize}
\item $d=0$:
  \begin{equation}
    \label{eq:6}
    T^0\in C_{G(\varepsilon_0),F(\varepsilon_0)}\text{ s.t. }\mu^1(T^0)=0
  \end{equation}
\item $d=1$:
    \begin{equation}
      \label{eq:2}
      \mu^2(G(a_1),T^0)+(-1)^{(g-1)(\vert a_1\vert -1)}\mu^2(T^0,F(a_1))+\mu^1(T^1(a_1))-T^1(\mu^1(a_1))=0
    \end{equation}

\item $d=2$:
    \begin{align}
      \label{eq:5}
     & \mu^3(G^1(a_2),G^1(a_1),T^0)+(-1)^{(g-1)(\vert a_1\vert -1)}\mu^3(G^1(a_2),T^0,F^1(a_1))\\
+& (-1)^{(g-1)(\vert a_1\vert+ \vert a_2\vert )}\mu^3(T^0,F^1(a_1),F^1(a_2))\nonumber\\
+&\mu^2(G^2(a_2,a_1),T^0)+(-1)^{(g-1)(\vert a_1\vert+\vert a_2\vert)}\mu^2(T^0,F^2(a_2,a_1))\nonumber\\
+&\mu^2(G^1(a_2),T^1(a_1))+(-1)^{(g-1)(\vert a_1\vert -1)}\mu^2(T^1(a_2),F^1(a_1))+\mu^1(T^2(a_2,a_1))\nonumber\\
+&(-1)^gT^2(a_2,\mu^1(a_1))+(-1)^{g+\vert a_1\vert-1}T^2(\mu^1(a_2),a_1)+(-1)^{g}T^1(\mu^2(a_2,a_1))=0\nonumber
    \end{align}
\end{itemize}

Equation \eqref{eq:6} implies that $T^0$ descends to a family of maps from $F(\varepsilon_0)$ to $G(\varepsilon_0)$ in the homological category ($[T^0_{\varepsilon_0}]\in H(C_{G(\varepsilon_0),F(\varepsilon_0)})$) which by \eqref{eq:2} satisfies $H(G^1)([a])\circ [T^0_{\varepsilon_0}]=(-1)^{g\cdot\vert a\vert}[T^0_{\varepsilon_0}]\circ H(F^1)([a])$ i.e. $[T^0]$ is a natural transformation from $H(F^1)$ to $H(G^1)$.

Now assume that $T^0=0$. Then equation \eqref{eq:2} implies that 
$T^1_{\varepsilon_1,\varepsilon_0}: C_{\varepsilon_1,\varepsilon_0}\rightarrow C_{G(\varepsilon_1),F(\varepsilon_0)}[g-1]$ is a chain map inducing $[T^1_{\varepsilon_1,\varepsilon_0}]$ in homology. 

Following \cite{Seidel_Fukaya} we denote by $\mathit{nu\textrm{-}fun}(\mathcal{F},\mathcal{G})^g$ the set of pre-natural transformations from $\mathcal{F}$ to $\mathcal{G}$. It comes with a filtration where $F^r(\mathit{nu\textrm{-}fun}(\mathcal{F},\mathcal{G})^g)$ are pre-natural transformations such that $T^k=0$ for $k\leq r$. It is obvious from Equation \eqref{eq:3} that $\mu^1$ preserves this filtration. The previous remarks about the $T^0=0$ case is part of the statement in \cite[Section (1f)]{Seidel_Fukaya} that the first page of the spectral sequence associated to this filtration is 
$$
E^{r,g}_1=\prod\limits_{\varepsilon_0,\ldots,\varepsilon_r}Hom^g_{\mathbb{K}}(Hom_{H(\mathcal{A})}(\varepsilon_{r-1},\varepsilon_{r})\otimes\ldots\otimes Hom_{H(\mathcal{A})}(\varepsilon_{1},\varepsilon_{0}),Hom_{H(\mathcal{B})}(F(\varepsilon_0),G(\varepsilon_r)).
$$

\paragraph{Yoneda modules.}
\label{sec:yonedas-module}

Recall from \cite{Seidel_Fukaya} that given an object $\varepsilon$ of an $\mathcal{A}_\infty$-category $\mathcal{A}$ one can see the assignment $\varepsilon_1\mapsto C_{\varepsilon,\varepsilon_1}$ as being part of the definition of a right $\mathcal{A}_\infty$-module over $\mathcal{A}$, called the right Yoneda module of $\varepsilon$ and denoted by $\mathcal{Y}_r^{\varepsilon}$. The action $\mu^d_{\mathcal{Y}_r^{\varepsilon}}: \mathcal{Y}_r^{\varepsilon}(\varepsilon_{d-1})\otimes C_{\varepsilon_{d-1},\varepsilon_{d-2}}\otimes\ldots\otimes C_{\varepsilon_1,\varepsilon_0}\rightarrow \mathcal{Y}_r^{\varepsilon}(\varepsilon_0)$ of $\mathcal{A}$ on $\mathcal{Y}_r^{\varepsilon}$ is given by the $\mathcal{A}_\infty$ operation $\mu^d$ (see \cite[Section (1l)]{Seidel_Fukaya}).  This assignment is part of an $\mathcal{A}_\infty$-functor, called the Yoneda embedding and denoted by $\mathcal{Y}_r$, from $\mathcal{A}$ to the $\mathcal{A}_\infty$-category of right modules over $\mathcal{A}$. This functor is cohomologically faithful, it however fails to be cohomologically full if the category $\mathcal{A}$ lacks a cohomological unit (one cannot construct the retraction of \cite[Section (2g)]{Seidel_Fukaya}).

Let $\varepsilon$ and $\varepsilon'$ be two objects of an $\mathcal{A}_\infty$-category $\mathcal{A}$, we say that $\varepsilon$ is \textit{pseudo-isomorphic} to $\varepsilon'$ if there exists a quasi-isomorphism of $\mathcal{A}_\infty$-modules $T: \mathcal{Y}_r^{\varepsilon}\rightarrow\mathcal{Y}_r^{\varepsilon'}$ (i.e. $T$ induces an isomorphism in homology, see Definition \ref{rem:yoneda_isom}). We will consider two $\mathcal{A}_\infty$-categories $\mathcal{A}$ and $\mathcal{B}$ to be \textit{pseudo-equivalent} if there exists a cohomologically full and faithful functor $\mathcal{F}:\mathcal{A}\rightarrow\mathcal{B}$ such that any object of $\mathcal{B}$ is pseudo-isomorphic to an object of the form $F(\varepsilon)$. In this situation $\mathcal{F}$ will be referred to as a \textit{pseudo-equivalence} of $\mathcal{A}_\infty$-categories. Of course, when both involved categories admit cohomological units the definition recovers the one from \cite{Seidel_Fukaya} since the Yoneda embedding is cohomologically full and the last part of the definition implies that $H(\mathcal{F})$ is essentially surjective. 

\paragraph{Cohomological adjunctions.}
\label{sec:pseudo-adjunctions}
In order to prove that stable isomorphic DGAs have pseudo-equivalent augmentation categories, one needs to prove that the relation of pseudo-equivalence is an equivalence relation. In order to do so we will introduce the notion of cohomologically adjoint functors and cohomological adjunction.

It is worthwhile to unravel the definition of morphism of modules in the $\mathcal{A}_\infty$ sense. A morphism of Yoneda modules is given by a sequence $T^d:\mathcal{Y}_r^{\varepsilon}(\varepsilon_{d-1})\otimes C_{\varepsilon_{d-1},\varepsilon_{d-2}}\otimes\ldots\otimes C_{\varepsilon_1,\varepsilon_0}\rightarrow\mathcal{Y}_r^{\varepsilon'}(\varepsilon_0)$, $d \ge 1$, i.e. a family of maps 
$$
T^d:C_{\varepsilon,\varepsilon_{d-1}}\otimes C_{\varepsilon_{d-1},\varepsilon_{d-2}}\otimes\ldots\otimes C_{\varepsilon_1,\varepsilon_0}\rightarrow C_{\varepsilon',\varepsilon_0},
$$
such that (see \cite[Equation (1.21)]{Seidel_Fukaya}):
\begin{align}\label{eq:7}
  &\sum_{j=0}^{d-1} (-1)^{\dagger}\mu^{j+1}(T^{d-j}(a_d,a_{d-1},\ldots,a_{j+1}),a_j,\ldots,a_1)\nonumber\\
+&\sum_{i=1}^d\sum_{j=0}^{d-i} (-1)^{\dagger}T^{d-i+1}(a_d,a_{d-1},\ldots,\mu^i(a_{j+i},\ldots,a_{j+1}),a_j,\ldots,a_1)=0.
\end{align}

For $d=1$ Equation \eqref{eq:7} implies that $T^1$ is a chain map, we recall now the definition of quasi-isomorphism.

\begin{defn}\label{rem:yoneda_isom}
  A morphism $T$ between two Yoneda modules $\mathcal{Y}_r^\varepsilon$ and $\mathcal{Y}_r^{\varepsilon'}$ is a \textit{quasi-isomorphism} if for every object $\varepsilon_0$ the chain-map $T^1_{\varepsilon_0}: C_{\varepsilon,\varepsilon_0}\rightarrow C_{\varepsilon',\varepsilon_0}$ induces an isomorphism in homology.
\end{defn}

Let $\mathcal{F}$ be a pseudo-equivalence from $\mathcal{A}$ to $\mathcal{B}$ then for any object $\varepsilon$ of $\mathcal{B}$ there is an object of $\mathcal{A}$ that we denote by $G(\varepsilon)$ such that $\varepsilon$ is pseudo-isomorphic to $F(G(\varepsilon))$. Considering all together the pseudo-isomorphism leads to a family of maps
$$T^d:C_{\varepsilon_d,\varepsilon_{d-1}}\otimes\ldots\otimes C_{\varepsilon_1,\varepsilon_0}\rightarrow C_{F\big(G(\varepsilon_d)\big),\varepsilon_0}.$$

It follows from \cite[Theorem 2.9]{Seidel_Fukaya} adapted in our context that $G$ is part of an cohomologically full and faithful functor $\mathcal{G}: \mathcal{B}\rightarrow \mathcal{A}$. Hence the family of maps $T^d$ is a pre-natural transformation $T\in F^1(\mathit{nu\textrm{-}fun}(Id,\mathcal{F}\circ\mathcal{G}))$. Comparing the equations \eqref{eq:3} and \eqref{eq:7} we get that this pre-natural transformation satisfies:
\begin{align}
  \mu^1(T)^d=\sum_{r=3}^d\sum_{s_1+\cdots +s_r=d} \!\!\!\!\!
  (-1)^\dagger &\mu_{\mathcal{B}}^r\big((\mathcal{F}\circ\mathcal{G})^{s_r}(a_d,\ldots,
  a_{d-s_r+1}), \ldots , (\mathcal{F}\circ\mathcal{G})^{s_{3}}(a_{s_1+s_2+s_{3}},\ldots,
  a_{s_1+s_2+1}),\nonumber\\
&T^{s_2}(a_{s_1+s_2},\ldots ,a_{s_1}),a_{s_1-1},\ldots,a_0\big). \label{eq:8}
  \end{align}

Note that since $\mathcal{F}$ is cohomologically full and faithful it admits an inverse $[F^1]^{-1}$ in homology, since we work over a field this map actually lifts to a chain map $U^1: C^{\mathcal{B}}_{F(\varepsilon'_1),F(\varepsilon'_0)}\rightarrow C^{\mathcal{A}}_{\varepsilon'_1,\varepsilon'_0}$. The composition $U^1_{G(\varepsilon_1),\varepsilon'_0}\circ T^1_{\varepsilon_1,F(\varepsilon'_0)}: C^{\mathcal{B}}_{\varepsilon_1,F(\varepsilon'_0)}\rightarrow C^{\mathcal{A}}_{G(\varepsilon_1),\varepsilon'_0}$ is a quasi-isomorphism inducing an isomorphism $\Theta:Hom(F(\varepsilon'_0),\varepsilon_1)\simeq Hom(\varepsilon'_0,G(\varepsilon_1))$. Combining equations \eqref{eq:5} (with $T^0=0$) and \eqref{eq:8} for the couple $(F^1(a),b)$ where $a\in C^{\mathcal{A}}_{\varepsilon_1',\varepsilon'_0}$ and $b\in C^\mathcal{B}_{\varepsilon_0,F(\varepsilon'_0)}$ one gets 

$$\mu_{\mathcal{B}}^2(T^1(b),F^1(a))+T^1(\mu_{\mathcal{A}}^2(b,F^1(a))=\mu^1(T^2) (b,F^1(a)).$$

This implies that $a$ induces the following commutative diagram in homology:

 \begin{align*}
\xymatrix{ {\text{Hom}(F(\varepsilon'_0),\varepsilon_0)} \ar[d]_{[\mu_{\mathcal{B}}^2(\cdot,F(a))]} \ar[r]^{\Theta} & {\text{Hom}(\varepsilon'_0,G(\varepsilon_0))} \ar[d]_{[\mu_{\mathcal{A}}^2(\cdot,a)]} \\
                    {\text{Hom}(F(\varepsilon'_1),\varepsilon_0)} \ar[r]^{\Theta}                  & {\text{Hom}(\varepsilon_1',G(\varepsilon_0))}                                          }
                    \end{align*}

A similar diagram exists for left composition. All in all, the fact that $T$ induces an isomorphism of Yoneda modules and that $\mathcal{F}$ and $\mathcal{G}$ are cohomologically full and faithful implies that $T^1$ induces an adjunction between $H(F)$ and $H(G)$. This justifies the following definition.

\begin{defn}\label{def:adjunction}
Let $\mathcal{F}:\mathcal{A}\rightarrow \mathcal{B}$ and $\mathcal{G}:\mathcal{B}\rightarrow\mathcal{A}$ be two cohomologically full and faithful functors. A \textit{cohomological adjunction} between $\mathcal{F}$ and $\mathcal{G}$ is a pre-natural transformation $T: Id\rightarrow \mathcal{F}\circ\mathcal{G}$ satisfying equation \eqref{eq:8}.
 \end{defn}

The previous discussion together with the adaptation of \cite[Theorem 2.9]{Seidel_Fukaya} in our context implies the following proposition. 

\begin{Prop}\label{sec:cohom-adjunctprop}
An $\mathcal{A}_\infty$-functor $\mathcal{F}: \mathcal{A}\rightarrow\mathcal{B}$ is a pseudo-equivalence iff there exists a cohomologically full and faithful functor $\mathcal{G}:\mathcal{B}\rightarrow\mathcal{A}$ and a cohomological adjunction $T$ between $\mathcal{F}$ and $\mathcal{G}$.
  \end{Prop}

We cannot emphasise enough that the fact that we work over a field is absolutely necessary for the last claim to be true, recall that in \cite{Seidel_Fukaya} one relies on the fact that we can split any complex as a sum of an acyclic complex and a complex with vanishing differential.

Note that the degree $1$ pre-natural transformation $T: Id\rightarrow Id$ defined by $T^1=Id$ and $T^d=0$ for $d\not= 1$ satisfies $\mu^1(T)^d=\mu^d$ for $d$ even and $0$ for $d$ odd and thus is a cohomological adjunction from $Id$ to $Id$. Under the geometric picture of Section \ref{sec:bilin-legendr-cont}, this should correspond to the interpolation between the two comparison maps given by the trivial cylinder which counts holomorphic curves with moving boundary conditions on (a small perturbation of) the trivial cylinder. In this situation only the trivial curves contribute which gives the definition of $T^1$.

Note also that in the presence of $c$-unit, the evaluation $T^1(c)$ is an element of $C_{G(F(\varepsilon_1)),\varepsilon_0}$ which induces a natural isomorphism from $H(\mathcal{G}\circ\mathcal{F})$ to the identity as one expects from a quasi-equivalence in this context.

We are now able to prove the following proposition.
\begin{Prop}\label{prop:cohom-adjunct}
The relation of pseudo-equivalence of $\mathcal{A}_\infty$-categories is an equivalence relation.
  \end{Prop}

  \begin{proof}
The identity functor is obviously a pseudo-equivalence as $T^1=Id$ induces an isomorphism from $\mathcal{Y}^r_\varepsilon$ to itself hence the relation is reflexive. It follows from Proposition \ref{sec:cohom-adjunctprop} that it is symmetric. In order to prove transitivity let $\mathcal{F}_0:\mathcal{A}\rightarrow\mathcal{B}$ and $\mathcal{G}_0:\mathcal{B}\rightarrow \mathcal{A}$ be cohomologically adjoint pseudo-equivalences. Similarly let $\mathcal{F}_1:\mathcal{B}\rightarrow\mathcal{C}$ be the pseudo-equivalence from $\mathcal{B}$ to $\mathcal{C}$ with cohomological adjoint $\mathcal{G}_1$. Denote by $T_0$ and $T_1$ the respective homological adjunctions.

$\mathcal{F}_1$ and $\mathcal{G}_1$ being cohomologically full and faithful, one gets that the map $H(F^1_1)\circ H(G^1_1)$ is an isomorphism. Again as we work over a field, this implies that there exists a chain map $S_{\varepsilon_0,\varepsilon_1}: C^{\mathcal{C}}_{F_1\big(G_1(\varepsilon_1)\big),F_1\big(G_1(\varepsilon_0)\big)}\rightarrow C^{\mathcal{C}}_{\varepsilon_1,\varepsilon_0}$ which induces $\big(H(F^1_1)\circ H(G^1_1)\big)^{-1}$ in homology.

This implies that the following composition
\begin{displaymath}
  \label{eq:14}
  \xymatrix { C^{\mathcal{C}}_{\varepsilon_1,\varepsilon_0} \ar[r]^-{T_0\circ G_1} & C^{\mathcal{B}}_{F_0\circ G_0\circ G_1(\varepsilon_1),G_1(\varepsilon_0)} \ar[r]^-{F_1}& {C^{\mathcal{C}}_{F_1\circ F_0\circ G_0\circ G_1(\varepsilon_1),F_1\circ G_1(\varepsilon_0)}} \ar[r]^-{S\circ T^1} & {C^{\mathcal{C}}_{F_1\circ F_0\circ G_0\circ G_1(\varepsilon_1),\varepsilon_0}}}
\end{displaymath}

is a quasi-isomorphism. This is the first application necessary to define an adjunction from $\mathcal{F}_1\circ \mathcal{F}_0$ to $\mathcal{G}_1\circ \mathcal{G}_0$. 

In order to define the higher order terms denote by $\mathcal{A}'$ the image of $\mathcal{F}^1\circ \mathcal{G}^1$. The homological perturbation lemma of \cite[Section (1i)]{Seidel_Fukaya} allows to extend $S$ to an $\mathcal{A}_\infty$-functor on $\mathcal{A}'$. The homological adjunction is thus given by $\mu^2(\mathcal{L}_S T^1,\mathcal{L}_{F_1}\mathcal{R}_{G_1}T_0)$ where $\mu^2$ is the composition of natural transformations, $\mathcal{L}$ and $\mathcal{R}$ are respectively the left and right composition of a pre-natural transformation with a functor (see \cite[Section (1e)]{Seidel_Fukaya}). We are brief here as in the present paper only the homological result will be relevant and thus only the first map matters.
  \end{proof}

\begin{Rem}\label{sec:comp-quasi-equiv-1}
  Note that at the most elementary level one gets that for an
  $\mathcal{A}_\infty$-category the set of isomorphism types of the
  groups $H(C_{\varepsilon_1,\varepsilon_0})$ is invariant under
  quasi-equivalences.
\end{Rem}

\begin{Rem}\label{sec:cohom-adjunctremleftright}
  Note that the map $S\circ T$ appearing in the proof of Proposition \ref{prop:cohom-adjunct} induces an isomorphism from $Hom(\varepsilon_1,F(\varepsilon'_0))$ to $Hom(G(\varepsilon_1),\varepsilon'_0)$ for any $\varepsilon'_0\in \mathcal{A}$ and $\varepsilon_1\in\mathcal{B}$ whenever $T$ is a homological adjunction. This justifies the ambiguity between left and right adjunction.
\end{Rem}

\subsection{Invariance}
\label{sec:invariance}

\paragraph{Construction of Functor on \textit{Aug(A)}.}
\label{sec:constr-funct-text}

Let $\mathcal{A}$ and $\mathcal{B}$ be two semi-free DGAs with generating sets $A$ and $B$ respectively and let $f: \mathcal{A}\rightarrow\mathcal{B}$ be a DGA map. We denote by $f_n$ the associated map between $\mathcal{A}_n$ and $\mathcal{B}_n$.

For an augmentation $\varepsilon$ set $F(\varepsilon)=\varepsilon\circ f$. 

For a $n$-tuple of augmentations $E = (\varepsilon_1, \ldots, \varepsilon_n)$ one defines $f_{E}:\mathcal{A}_n\rightarrow\mathcal{B}_n$ by the $\mathbb{K}_n$-algebra morphism $\phi_{\varepsilon_E}\circ f_n \circ\phi_{F(\varepsilon)_E}^{-1}$. This is a DGA-map from $(\mathcal{A}_n,\partial^{\varepsilon_E\circ f})\rightarrow (\mathcal{A},\partial^{\varepsilon_E})$. Its restriction to $C_n(A)$ decomposes as a sum $f_1^{\varepsilon_E}\oplus f_2^{\varepsilon_E}\oplus\ldots\oplus f_j^{\varepsilon_E}\oplus\ldots$ where $f_j^{\varepsilon_E}:C_n(A)\rightarrow C_n(B)^{\otimes j}$ are $\mathbb{K}_n$-bimodule homomorphisms. Note that there is no $0$-order term in the decomposition as those coming from $F_n(\varepsilon_E)$ cancel with those coming from $\phi_{{\varepsilon_E}\circ f_n}^{-1}$. Dualising each of the $f_j^{\varepsilon_E}$ and restricting them to $
C_{\varepsilon_d,\varepsilon_{d-1}} \otimes \ldots \otimes C_{\varepsilon_2, \varepsilon_1}$ 
one gets maps $F^d_{\varepsilon_E}:C_{\varepsilon_d,\varepsilon_{d-1}} \otimes \ldots \otimes C_{\varepsilon_2, \varepsilon_1}\rightarrow C_{F(\varepsilon_d),F(\varepsilon_1)}[1-d]$ (again note that here we use the absolute grading). As in Section \ref{sec:augm-categ-1} those only depend on the augmentations $\varepsilon_1, \ldots, \varepsilon_d$ and on $f$. We denote by $F^d$ the sets of maps $\{F^d_{\varepsilon_d, \ldots, \varepsilon_1}\}$ for all $d$-tuples of augmentations. From $f^{\varepsilon_E}\circ \partial^{F(\varepsilon_E)}=\partial^{\varepsilon_E}\circ f^{\varepsilon_E}$ one deduces the following

\begin{Prop}\label{thm:constr-funct-text}
  The family $\mathcal{F}=\{F^d\}$ is an $A_\infty$-functor from $\mathit{Aug(B)}$ to $\textit{Aug(A)}$.
\end{Prop}

The construction is functorial with respect to the composition as stated in the following 
\begin{Prop}
  Let $f$ and $g$ be two DGA maps then the $\mathcal{A}_\infty$-functor associated to $f\circ g$ is $\mathcal{G}\circ \mathcal{F}$, where the composition of $\mathcal{A}_\infty$-functors is defined as in \cite[Section (1b)]{Seidel_Fukaya}. 
\end{Prop}

\begin{proof}
  Consider $(f\circ g)^{\varepsilon_E}\vert _{C(A)}$. As $f$ is a DGA map one gets the decomposition
$$(f\circ g)^{\varepsilon_E}\vert _{C(A)}(a)=\sum \sum f_{s_1}\otimes\cdots  \otimes f_{s_r} (g_i(a)).$$
 Dualising this equation and denoting $\mathcal{H}$ the functor associated to $f\circ g$ leads to the formula
$$H^d(a_d,\cdots,a_1)=\sum\limits_r\sum\limits_{s_1+\ldots+s_r=d}G^r(F^{s_r}(a_d,\ldots,a_{d-s_r+1}),\ldots, F^{s_1}(a_{s_1},\ldots,a_1))$$
which is the composition formula of $\mathcal{A}_\infty$-functors.
\end{proof}

\paragraph{Equivalence of augmentations.}
\label{sec:equiv-augm}
We are now ready to give the definition of equivalent augmentations.
\begin{defn}
  Two augmentations $\varepsilon_1$ and $\varepsilon_2$ are equivalent if there is an $\mathcal{A}_\infty$-functor $\mathcal{F}$ such that $\varepsilon_1=F(\varepsilon_2)$ and a homological adjunction $T$ from $F$ to the identity.
\end{defn}

Note that in \cite{Bourgeois_Survey} equivalence of augmentation has a different shape. Precisely it is said that two augmentations are equivalent if there exists a derivation $K$ such that $\varepsilon_1=\varepsilon_2\circ e^{K\partial+\partial K}$. This definition is of course problematic if our coefficient ring is not an algebra over a field of characteristic $0$. Equivalent augmentations from \cite{Bourgeois_Survey} are equivalent in our sense using $f=e^{K\partial+\partial K}$ to construct the 
$\mathcal{A}_\infty$-functor. The natural transformation is given by dualising homogenous component of $K$. It is not clear to us wether or not the converse is true, compare also Formula \eqref{eq:7} with the formula appearing in \cite[Lemma 3.13]{Kalman_&_Lagrangian_Cobordism}.

Definition \ref{rem:yoneda_isom} and \ref{sec:cohom-adjunctremleftright} imply the following

\begin{Thm}\label{thm:equivaug}
  Let $\varepsilon_1$ and $\varepsilon_2$ be two augmentations of $\mathcal{A}$. If $\varepsilon_1$ is equivalent to $\varepsilon_2$ then for any augmentation $\varepsilon$, we have $H(C_{\varepsilon_1,\varepsilon})\simeq H(C_{\varepsilon_2,\varepsilon})$ and $H(C_{\varepsilon,\varepsilon_1})\simeq H(C_{\varepsilon,\varepsilon_2})$.
\end{Thm}

\paragraph{Invariance under stable isomorphisms.}
\label{sec:invar-under-homot}

Recall that a stabilisation of a DGA is a new DGA with two more generators $b,c$ such that $\partial b=c$. Two DGAs are said to be \textit{stable isomorphic} if they become isomorphic after a sequence of stabilisations. The aim of this paragraph is to prove the following:

\begin{Thm}\label{thm:quasiisomostabletame}
  Let $\mathcal{A}$ and $\mathcal{A}'$ be two stable isomorphic DGAs. Then their augmentation categories are pseudo-equivalent.
\end{Thm}

\begin{proof}
  From Proposition \ref{prop:cohom-adjunct}, it is sufficient to prove the theorem first for an isomorphism and second for a stabilisation.
Let $f: \mathcal{A}\rightarrow\mathcal{A}$ be an isomorphism. Consider $\mathcal{F}$ the associated functor as in Proposition \ref{thm:constr-funct-text}, and $\mathcal{G}$ the functor associated to $f^{-1}$. Then $\mathcal{F}\circ \mathcal{G}=Id$ and $\mathcal{G}\circ \mathcal{F}=Id$ imply that the categories $\mathit{Aug}(\mathcal{A};\partial)$ and $\mathit{Aug}(\mathcal{A},\partial')$ are pseudo-equivalent.

Let $\mathcal{A}'$ be a stabilisation of $\mathcal{A}$. The natural inclusion $i:\mathcal{A}\rightarrow\mathcal{A}'$ is a DGA morphism, inducing a functor $\mathcal{I}$. Note that here $I^k=0$ for $k\geq 2$. The map $j:\mathcal{A}'\rightarrow \mathcal{A}$ which sends $b$ and $c$ to $0$ is also a DGA morphism inducing $\mathcal{J}$. Obviously $\mathcal{I}$ and $\mathcal{J}$ are cohomologically full and faithful and $\mathcal{J}\circ \mathcal{I}$ is the identity. Note that as vector spaces the morphism spaces $C^{\varepsilon_1,\varepsilon_2}$ and $C^{\varepsilon_1,I\circ J(\varepsilon_2)}$ are canonically the same and that the augmented differential $\mu^1_{\varepsilon_1,\varepsilon_2}$ corresponds to the one for $I\circ J (\varepsilon_i)$. This implies that the pre-natural transformation defined under this identification by $T^1_{\varepsilon_1,\varepsilon_2}=Id$ and $T^d=0$ if $d\not= 1$ is a homological adjunction from $\mathcal{I}$ to $\mathcal{J}$ as noted in Section \ref{sec:equiv-categ}.
\end{proof}

The set of equivalence classes of augmentations behave nicely with respect to stable isomorphisms of DGAs. Namely

\begin{Thm}
Let $\mathcal{A}(A)$ be a free DGA. Then we have:
  \begin{enumerate}
  \item {Let $f$ be a isomorphism of $\mathcal{A}$ then $\varepsilon_1$ is equivalent to $\varepsilon_2$ iff $\varepsilon_1\circ f$ is equivalent to $\varepsilon_2\circ f$.}
\item {Let $i: \mathcal{A}\rightarrow\mathcal{A}'$ be the inclusion of a DGA into one of his stabilisation. Then $\varepsilon_1$ is equivalent to $\varepsilon_2$ iff $\varepsilon_1\circ i$ is equivalent to $\varepsilon_2\circ i$ and two augmentation $\varepsilon'_1,\varepsilon_2'$ of $\mathcal{A}'$ are equivalent iff their restriction to $\mathcal{A}$ are.}
  \end{enumerate}
\end{Thm}

\begin{proof}
The first part is obvious. It is sufficient to compose the natural adjunction with the functor induced by $f^{-1}$.

If the stabilisation is not of degree $0$ the result is obvious, thus we suppose that the stabilisation has degree $0$. One needs to prove that if $\varepsilon'$ is an augmentation of $\mathcal{A}'$ then $\varepsilon_0'$ defined by $\varepsilon_0'(a)= \varepsilon'(a)$ for all generators $a \in A$ and $\varepsilon_0'(b) = \varepsilon_0'(c)= 0$ is equivalent to $\varepsilon'$. To do so we proceed as in the proof of Theorem \ref{thm:quasiisomostabletame}. We define the DGA morphism $f:\mathcal{A}'\rightarrow\mathcal{A}'$ by $f(a)=a$ for all $a\in A$, $f(b)=b-\varepsilon'(b)$ and $f(c)=c$. Obviously $F(\varepsilon')=\varepsilon_0'$. The pre-natural transformation $T$ defined by $T^1=Id$ and $T^d=0$ induces the desired homological natural adjunction. The proof is complete as $\varepsilon_0'=\varepsilon'\vert_{\mathcal{A}}\circ i$.
\end{proof}

\section{Geometric interpretation}

\subsection{Legendrian contact homology}

In this section we consider $\mathcal{J}^1(M):=T^*M\times\mathbb{R}$ the jet space of a $n$-dimensional manifold $M$ with the standard contact structure $\xi=\ker(dz-\theta)$ where $z$ parametrises $\mathbb{R}$  and $\theta=\sum pdq$ is the standard Liouville form on $T^*M$ with $(q,p)$ the conjugate coordinates. A Legendrian submanifold $\Lambda$ of $\mathcal{J}^1(M)$ is a $n$-dimensional submanifold such that $T\Lambda\subset\xi$. We assume here that all Legendrian submanifolds are compact. 

We denote by $\pi$ the canonical projection of $\mathcal{J}^1(M)$ to $T^*M$ and by $\Pi$ the canonical projection of $\mathcal{J}^1(M)$ to $M\times\mathbb{R}$. For a Legendrian submanifold $\Lambda$, $\pi(\Lambda)$ is called the \textit{Lagrangian projection} and $\Pi(\Lambda)$ is called the \textit{front projection}.

A Reeb chord $\gamma$ of $\Lambda$ is trajectory $\gamma: [0,T]\rightarrow \mathcal{J}^1(M)$ of $\frac{\partial}{\partial z}$ such that $T>0$, $\gamma(0)$ and $\gamma(T)$ belong to $\Lambda$. Those are in bijection with double points of the Lagrangian projection. A Reeb chord is \textit{non-degenerate} if the corresponding double point of the Lagrangian projection is transverse and if all Reeb chords are non-degenerate $\Lambda$ is called \textit{chord generic}. We denote by $\mathcal{R}(\Lambda)$ the set of Reeb chords of $\Lambda$ (if $\Lambda$ is chord generic this is a finite set since $\Lambda$ is compact).

If $\Lambda$ is chord-generic there is a grading map $gr:\mathcal{R}(\Lambda)\rightarrow\mathbb{Z}$ defined by the Conley-Zehnder index (see \cite{Ekholm_Contact_Homology}).

We denote by $C(\mathcal{R}(\Lambda))$ the graded $\mathbb{K}$-vector space generated by $\mathcal{R}(\Lambda)$ and $\mathcal{A}(\Lambda)$ the tensor algebra of $C(\mathcal{R}(\Lambda))$ as in Section \ref{sec:algebraic-setup}. In general $\mathbb{K}$ will be $\mathbb{Z}_2$; if $\Lambda$ is relatively spin, one can also consider $\mathbb{Q}$. This algebra is called the Chekanov algebra of $\Lambda$, it is a differential graded algebra with differential $\partial$ defined counting some holomorphic curves in $\mathbb{R}\times \mathcal{J}^1(M)$ (see \cite{LCHgeneral} and \cite{Ekholm_&_Orientation_Homology}). 

Let $\widetilde{J}$ be an almost complex structure on $T^*M$ compatible with $-d\theta$. To $\widetilde{J}$ we associate an almost complex structure $J$ on $\mathbb{R}\times \mathcal{J}^1(M)$ by the following. The differential of the projection $\pi$ induces an isomorphism between $\xi\vert_{(q,p,z)}$ and $T_{(q,p)}(T^*M)$; we set $J\vert_\xi=\widetilde{J}$ under this identification. Finally we set $J\frac{\partial}{\partial t}$ to be equal to $\frac{\partial}{\partial z}$. Such an almost complex structure will be referred to as a \textit{compatible} almost complex structure.  

We denote by $D_k$ the $2$-dimensional closed unit disk with the $(k+1)$-th roots of unity removed.  For a complex structure $j$ on $D_k$ we choose holomorphic coordinates $[T,\infty]\times [0,1]$ near $1$ and $[-\infty,-T]\times [0,1]$ near the other punctures. A map $u: (D_k,j)\rightarrow \mathbb{R}\times\mathcal{J}^1(M)$ is said to be holomorphic if $du\circ j=J\circ du$. 

If $u$ is a holomorphic map then $u\vert_{[T,\infty]\times [0,1]}$ decomposes as a map $(a,v,f)$ with $a:[T,\infty]\times [0,1]\rightarrow \mathbb{R}$, $v:[T,\infty]\times [0,1]\rightarrow T^*M$ and $f:[T,\infty]\times [0,1]\rightarrow \mathbb{R}$. Suppose $u(\partial{D_{k}})\subset \mathbb{R}\times \overline{\Lambda}_n^t$ then for a Reeb chord $\gamma$ we say that $u$ has a positive asymptotic $\gamma$ at $1$ if $v(z)\rightarrow \pi(\gamma)$ and $a(z)\rightarrow\infty$ when $z\rightarrow 1$. 
Similarly, for a root of unity $z_0$ we say that $u$ has a negative asymptotic $\gamma$ at $z_0$ if $v(z)\rightarrow \pi(\gamma)$ and $a(z)\rightarrow -\infty$ when $z\rightarrow z_0$.

For Reeb chords $\gamma^+$, $\gamma_1,\ldots,\gamma_k$ we denote by $\mathcal{M}(\gamma^+,\gamma_1,\ldots,\gamma_k)$ the moduli space of holomorphic maps from $(D_k,j)$ to $\mathbb{R}\times \mathcal{J}^1(M)$ with boundary on $\mathbb{R}\times\Lambda$, positive asymptotics $\gamma^+$ and negative asymptotics $\gamma_1,\ldots, \gamma_k$ for all $j$ modulo biholomorphism $(D_k,j)\simeq (D_{k},j')$ and translation in the $\mathbb{R}$ direction in $\mathbb{R}\times \mathcal{J}^1(M)$.

The differential $\partial$ on $\mathcal{A}(\Lambda)$ is defined by
$$
\partial \gamma^+ =  \sum_{\stackrel{\gamma_1, \ldots, \gamma_k}{gr(\gamma_1\ldots \gamma_k) = gr(\gamma^+)-1}} 
\# \mathcal{M}(\gamma^+,\gamma_1,\ldots,\gamma_k) \cdot \gamma_1 \ldots \gamma_k 
$$
on generators and is extended to $\mathcal{A}(\Lambda)$ by linearity and the Leibniz rule.

The homology of the DGA $(\mathcal{A}(\Lambda), \partial)$  is called the \textit{Legendrian contact homology} of $\Lambda$ and is denoted by $LCH(\Lambda)$.

As the Chekanov algebra (see \cite{Chekanov_DGA_Legendrian} and \cite{Ekholm_Contact_Homology}) of a Legendrian submanifold $\Lambda$ is a semi-free DGA, the previous construction applies in this case to give the augmentation category for a Legendrian submanifold. We denote this category by $\mathit{Aug}(\Lambda)$.  In the next section, we will show that this augmentation category can be extracted from the geometrical data of the differential of the Chekanov algebra of the $n$-copy Legendrian link. 

One can choose more sophisticated coefficient rings to define Legendrian contact homology, for instance one can keep track of the homology class of the curves defining the differential using the group ring $\mathbb{K}[H_1(\Lambda)]$. Even though this is not a field, one can still carry out the previous construction considering all the chain complexes, tensor products and dual spaces over $\mathbb{K}$ (the coefficient from $H_1(\Lambda)$ inducing a decomposition of the vector spaces).

\subsection{The augmentation category of Legendrian submanifolds}
\label{sec:bilin-legendr-cont}

\paragraph{The augmentation category \boldmath$\mathit{Aug}(\Lambda)$.}
\label{sec:augm-mathc-1}

From Section \ref{sec:algebraic-setup}, we deduce that there is an $\mathcal{A}_\infty$-category associated to $\mathcal{A}(\Lambda)$ we denote this category by $\mathit{Aug}(\Lambda)$. Note that from Theorem \ref{thm:equiv} the curve contributing to the products $\mu^n$ are the curves shown in Figure \ref{fig:comp}.

The homologies of the morphisms $(C(\mathcal{R}(\Lambda)),d^{\varepsilon_0,\varepsilon_1})$ and $(C(\mathcal{R}(L)),\mu^1_{\varepsilon_1,\varepsilon_0})$ are called the \textit{bilinearised Legendrian contact homology and cohomology group}, and are denoted by $LCH^{\varepsilon_0,\varepsilon_1}(L)$ and $LCH_{\varepsilon_1,\varepsilon_0}(L)$.

Recall that Legendrian isotopies induce stable tame isomorphisms (and thus stable isomorphisms) of the Chekanov algebra (see \cite{Chekanov_DGA_Legendrian}, \cite{Ekholm_Contact_Homology} and \cite{LCHgeneral}). Thus Theorem \ref{thm:quasiisomostabletame} implies Theorems \ref{thm:quasiequilegisot} and \ref{thm:legisotinv}. Theorem \ref{thm:equiclassisot} is itself a consequence of Theorem \ref{thm:equivaug} which also implies Theorem \ref{thm:equivaugisom}.

We will now describe how the bilinearised differential appears naturally when considering the Chekanov algebra of the two copy Legendrian link associated to $\Lambda$.

\paragraph{The \boldmath$n$-copy Legendrian link.}
\label{sec:n-copy-legendrian}

Let $\Lambda$ be a Legendrian submanifold of $\mathcal{J}^1(M)$. Choose Morse functions $f_i, i=2, \ldots, n$ on $\Lambda$ such that, for each $i \neq j$, the critical points of $f_i$ and $f_j$ are disjoint and $f_i-f_j$ is a Morse function. We also set $f_1=0$ in order to simplify notations.  
We denote by $\overline{\Lambda}_n$ the Legendrian link $\Lambda_1\sqcup \Lambda_2\sqcup\ldots \sqcup \Lambda_n$ where $\Lambda_i$ is a perturbation of 
$\Lambda +i\varepsilon \frac{\partial}{\partial z}$ by the $1$-jet of the function $\varepsilon f_i$.

\paragraph{Reeb chords of \boldmath$\Lambda_n$.}
\label{sec:reeb-chords-}

Let $\mathcal{R}(\Lambda)$ be the set of Reeb chords of $\Lambda$. Reeb chords of $\overline{\Lambda}_n$ are of three different types:
\begin{enumerate}
\item {Reeb chords of $\Lambda_i$ for each $i\in \{1,\ldots, n\}$ called \textit{pure chords}.}
\item {Critical points of $f_i-f_j$ for each $i\not= j$ called \textit{continuum chords}.}
\item {Long (i.e. not continuum) Reeb chords from $\Lambda_i$ to $\Lambda_j$ for each $i\not= j$ called 
\textit{mixed chords}.}
\end{enumerate}

For a fixed $i$, the set of chords of type $1$ of $\Lambda_i$ are in bijection with $\mathcal{R}(\Lambda)$; for $\gamma\in \mathcal{R}(\Lambda)$ we denote $\gamma_{i,i}$ the corresponding chord of $\Lambda_i$. Similarly, for each $i\not= j$, chords of type $2$ are in bijection with $\mathcal{R}(\Lambda)$ and for $\gamma\in\mathcal{R}(\Lambda)$ we denote by $\gamma_{i,j}$ the 
corresponding chord from $\Lambda_i$ to $\Lambda_j$.

\begin{figure}[ht!]
\labellist
\small\hair 2pt
\pinlabel {$\gamma_{1,1}$} [tl] at 186 23
\pinlabel {$\gamma_{2,2}$} [tl] at 194 137
\pinlabel {$\gamma_{1,2}$} [tl] at 152 135
\pinlabel {$\gamma_{2,1}$} [tl] at 146 79
\endlabellist
  \centering
  \includegraphics[height=5cm]{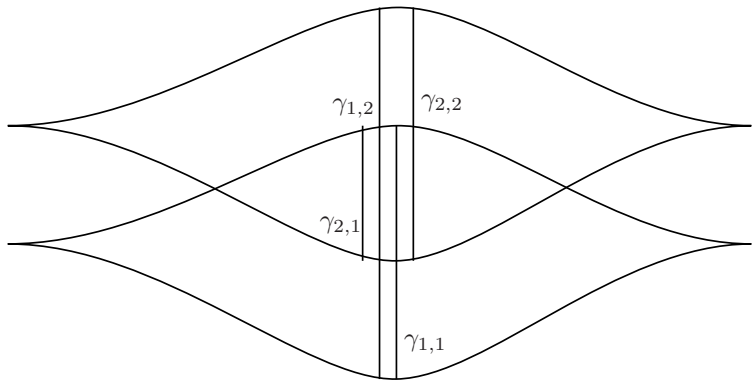}
  \caption{Front projection of the $2$-copy Legendrian unknot before perturbation.}
  \label{fig:front}
\end{figure}

\begin{figure}[ht!]
\labellist
\small\hair 2pt
\pinlabel {$\gamma_{1,1}$} [tl] at 157 103
\pinlabel {$\gamma_{2,2}$} [tl] at 157 208
\pinlabel {$\gamma_{1,2}$} [tl] at 110 155
\pinlabel {$\gamma_{2,1}$} [tl] at 208 155
\pinlabel {$M$} [tl] at 356 155
\pinlabel {$m$} [tl] at 24 155
\endlabellist
  \centering
  \includegraphics[height=5cm]{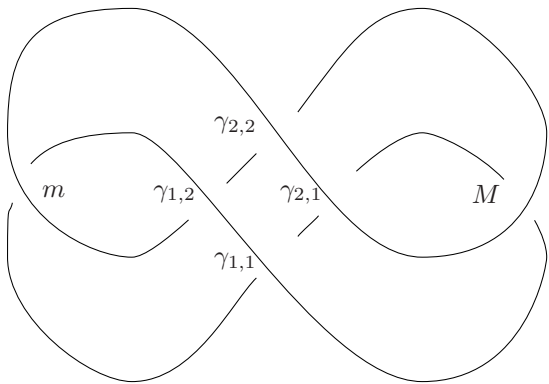}
  \caption{Lagrangian projection of the $2$-copy Legendrian unknot after perturbation.}
  \label{fig:lag}
\end{figure}

\paragraph{The differential.}
\label{sec:differential}
Let $I_c$ be the two-sided ideal of $\mathcal{A}(\overline{\Lambda}_n)$ generated by continuum
chords, i.e. Reeb chords coming from critical points of the functions $f_i-f_j$. Then we have:

\begin{Prop}\label{sec:differential-1propideal}
  For sufficiently small $\varepsilon>0$, $\partial(I_c)\subset I_c$.
\end{Prop}

\begin{proof}
  Let $\ell$ be the minimal length of all Reeb chords of type $1$ and $2$ and $F$ be the maximum 
  of all maxima of the functions $f_i-f_j$ so that the length of any chord of type $3$ is smaller than $\varepsilon (n-1 + F)$. Set $\varepsilon<\frac{\ell}{n-1+F}$. As $\partial$ decreases the length of chords, we get that for any chord of type $3$ its boundary is in the algebra generated by chords of type $3$. The Leibniz rule thus implies that $\partial (I_c)\subset I_c$.
\end{proof}

We denote by $\mathcal{A}_n(\Lambda)$ the quotient $\mathcal{A}(\overline{\Lambda}_n)/I_c$. Then the differential on $\mathcal{A}(\overline{\Lambda}_n)$ descends to a differential $ \partial_n$ on $\mathcal{A}_n(\Lambda)$. Note that  $\mathcal{A}_n(\Lambda)$ is isomorphic to the tensor algebra over $\mathbb{K}$ of $C_n(\Lambda)$ where $C_n(\Lambda)$ is the free $\mathbb{K}$-vector space generated by Reeb chords of $\overline{\Lambda}_n$ of type $1$ and $2$.

We extend the vector space structure of $C_n(\Lambda)$ to a $\mathbb{K}_n$-module structure by setting $e_i\cdot \gamma_{j,k}\cdot e_l=\delta_{ij}\cdot \delta_{kl} \cdot \gamma_{j,k}$.  

\paragraph{Structure of \boldmath$\partial_n$.}
\label{sec:structure-delta_n}

Consider now the subalgebra of $\mathcal{A}_n(\Lambda)$ linearly generated by $\gamma_{i_1,j_1}\otimes \ldots\otimes\gamma_{i_k,j_k}$ where $(I=(i_1, \ldots, i_k),J=(j_1, \ldots, j_k))$ is a composable pair of multi-indices. As a $\mathbb{K}_n$-algebra it is isomorphic to the algebra $\mathcal{A}_n(A)$ of Section \ref{sec:n-copy-algebra} with $A = \mathcal{R}(\Lambda)$. As the differential $\partial$ counts connected holomorphic curves it is obvious that 
$\partial_n(\mathcal{A}_n(\mathcal{R}(\Lambda)))\subset \mathcal{A}_n(\mathcal{R}(\Lambda))$.

\begin{Thm}\label{thm:equiv}\label{thm:dn=d12}
For sufficiently small $\varepsilon>0$, $\partial_n\vert_{\mathcal{A}_n(\mathcal{R}(\Lambda))}$ is the differential on the $n$-copy algebra of $(\mathcal{A}(\Lambda),\partial)$ as in Section \ref{sec:n-copy-algebra}.
\end{Thm}

\begin{proof}
Consider $\overline{\Lambda}_n^0$ the $n$-copy Legendrian link consisting of $n$-linked copies of $\Lambda$ translated in the $z$ direction, without perturbing them by Morse functions. Reeb chords of this link are not all non-degenerate as any point of $\Lambda_i$ belongs to a chord running from $\Lambda_i$ to $\Lambda_j$ for all $j\not= i$. However if chords of $\Lambda$ are non-degenerate then chords of $\overline{\Lambda}_n^0$ are Morse-Bott. Similarly to \cite{Bourgeois_Morse_Bott}, one can consider the Morse-Bott approach to Legendrian contact homology using a Morse-Bott generic compatible almost complex structure $J$ on $\mathbb{R}\times M$.

For sufficiently small $\varepsilon>0$, $J$ will be a generic almost complex structure in the sense of \cite{LCHgeneral} for the link $\overline{\Lambda}_n^t$ obtained by perturbing $\overline{\Lambda}^0_n$ using the Morse function $t\cdot f_i$, for all $t \in (0,\epsilon)$. All chords of $\overline{\Lambda}^t_n$ are described as before and are therefore in natural bijection for all $t \in (0,\epsilon)$. When $t=0$, all non-continuum chords are in bijection with non-continuum chords of $\overline{\Lambda}^t_n$ for $t>0$ and are non-degenerate. Thus, as they are all in natural bijection, we drop the index $t$ and denote them by the same letter.

Let $\mathcal{M}_{t}(\gamma^+,\gamma_1,\ldots,\gamma_k)$ be the moduli space of holomorphic
maps on $D_k$ with boundary on $\mathbb{R} \times \overline{\Lambda}_n^t$. We label the connected components of $\partial(D_k)$ by $C_0, \ldots, C_{k+1}$ starting from the arc connecting $1$ to $e^{2\pi i/(k+1)}$.
For a multi-index $I=(i_0,i_1,\ldots, i_{k+1})$ we denote by $\mathcal{M}^I_{t}(\gamma^+,\gamma_1,\cdots,\gamma_k)$ the moduli space consisting of holomorphic maps such that 
the connected component $C_j$ of $\partial(D_k)$ is mapped to $\mathbb{R} \times \Lambda^t_{i_j}$, 
for $j=0, \ldots, k+1$.  

If $\varepsilon >0$ is small enough, if $\gamma^+,\gamma_1,\ldots,\gamma_k$ are non-continuum chords, then the moduli spaces $\mathcal{M}^I_{t}(\gamma^+,\gamma_1,\ldots,\gamma_k)$ are diffeomorphic for all $t \in [0, \varepsilon)$, since all non-continuum chords are non-degenerate.

Recall that elements of $\mathcal{M}^I(\gamma^+,\gamma_1,\cdots,\gamma_j)$ are in bijection with holomorphic curves in $T^*M$ with boundary on the Lagrangian projection of $\Lambda_i$ with punctures to the double points corresponding to the Reeb chords (see \cite{1024.57014}). As the Lagrangian projection of $\Lambda+i\varepsilon \frac{\partial}{\partial z}$ is equal to the Lagrangian projection of $\Lambda$, the moduli space of holomorphic curves $\mathcal{M}^I(\gamma^+,\gamma_1,\ldots,\gamma_k)$ and $\mathcal{M}^J(\gamma^+,\gamma_1,\ldots,\gamma_k)$ are in bijective correspondence for any multi-indices $I$ and $J$ of length $k+2$.
\end{proof}

For a $n$-tuple of augmentations $E=(\varepsilon_1,\ldots, \varepsilon_n)$ of $\mathcal{A}(\Lambda)$ over $\mathbb{K}$ one gets an augmentation $\varepsilon_E$ of $\mathcal{A}(\overline{\Lambda}_n)$ over $\mathbb{K}_n$ setting $\varepsilon_E(\gamma_{i,i})=\varepsilon_i(\gamma) \cdot e_i$ for any
$\gamma \in \mathcal{R}(\Lambda)$ and sending every mixed chord to $0$ similarly as done in Section \ref{sec:augm-mathc}. The fact that this is an augmentation follows from Theorem \ref{thm:dn=d12}, Proposition \ref{sec:differential-1propideal} and applying Proposition \ref{sec:augm-mathc-prop} to the quotient algebra. 

For $n=2$, the augmented complex decomposes as $C^{1,1}\oplus C^{2,2}\oplus C^{1,2}\oplus C^{2,1}\oplus C^{1,2}(f)$ and from the fact that only the pure chords are augmented we get that $C^{i,i}$ are subcomplexes, the corresponding differential is the standard augmented differential. $C^{2,1}$ is also a subcomplex, it follows from Theorem \ref{thm:dn=d12} that its differential corresponds to the bilinearised differential $d^{\varepsilon_1,\varepsilon_2}$. Finally $C^{1,2}\oplus C^{1,2}(f)$ is a subcomplex with differential
$$
\left(\begin{array}{cc}
  d^{\varepsilon_1,\varepsilon_2} & 0\\
\rho & d^f
\end{array}\right) .
$$

The differential $d^{\varepsilon_1,\varepsilon_2}$ is the bilinearised differential and the complex $(C(f),d^f)$ is the Morse complex of $f_2$. This can be seen by a standard Morse-Bott argument and follows for example from \cite[Theorem 3.6]{Duality_EkholmetAl}. Finally $\rho$ is a chain map and thus the complex $C^{1,2}\oplus C^{1,2}(\rho)$ is the cone over $\rho$, we denote it $C(\rho)$.

\subsection{Duality exact sequence}
\label{sec:prod-struct-adn}

A Legendrian submanifold $\Lambda$ of $\mathcal{J}^1(M)$ is \textit{horizontally displaceable} if there exists a Hamiltonian diffeomorphism $\phi$ of $T^*M$ such that $\phi(\pi(\Lambda))\cap \pi(\Lambda)=\emptyset$. Note that if $M$ is non-compact any compact Legendrian submanifold is horizontally displaceable.

In \cite{Duality_EkholmetAl}, it is proved that linearised Legendrian contact homology and cohomology of horizontally displaceable Legendrian submanifolds are parts of an exact sequence (which generalise the duality of \cite{Sabloff_Duality}):

\begin{equation}
\cdots \rightarrow H_{k+1}(\Lambda)\rightarrow LCH_\varepsilon^{n-k-1}(\Lambda)\rightarrow LCH^\varepsilon_k(\Lambda)\rightarrow H_k(\Lambda)\rightarrow\cdots \label{eq:10}
\end{equation}
called the duality exact sequence.

The same considerations allow us to prove Theorem \ref{thm:duality}.

\begin{proof}[Proof of Theorem \ref{thm:duality}]
  This is almost a direct application of the arguments in \cite{Duality_EkholmetAl} applied to the 
  $4$-component link $2 \overline{\Lambda}_2$ which consists of two copies $\overline{\Lambda}_2$ and $\widetilde{\Lambda}_2$ of $\overline{\Lambda}_2$ far apart in the $z$-direction (one perturbs the second copy using Morse functions in the same manner as in the previous section).

First note that the discussion at the end of the previous section implies that there is an exact sequence induced by $\rho$:

$$\cdots\rightarrow H_{k+1}(\Lambda)\xrightarrow{i_*} H(C(\rho))\xrightarrow{\pi_*} LCH^{\varepsilon_1,\varepsilon_2}_k(\Lambda)\xrightarrow{\rho}\cdots$$

It remains to identify $H(C(\rho))$ with the bilinearised Legendrian cohomology. 
In order to do so, one must investigate how the exact sequence is built in \cite{Duality_EkholmetAl}. One closely follows the notation from here. The Reeb chords of $2 \overline{\Lambda}_2$ are of $4$ types:
\begin{enumerate}
\item {Chords going from $\widetilde{\Lambda}_2$ to itself denoted by $q^0$ and their corresponding chords on $\widetilde{\Lambda}_2$ denoted by $\widetilde{q^0}$.}
\item {Chords going from $\overline{\Lambda}_2$ to $\widetilde{\Lambda_2}$ denoted by $q^1$ whose positive and negative ends are nearby the positive and negative ends of the corresponding chords on $\overline{\Lambda}_2$.}
\item {Chords going from $\overline{\Lambda}_2$ to $\widetilde{\Lambda_2}$ denoted by $p^1$ whose positive are nearby the negative ends of the corresponding chords on $\widetilde{\Lambda_2}$ and negative ends are nearby the positive ends of the corresponding chords on $\overline{\Lambda}_2$.}
\item {Critical points of the Morse functions on $\Lambda$ used to perturb $\widetilde{\Lambda}_2$.}
\end{enumerate}

The corresponding vector space complexes are denoted $Q^0$, $Q^1$, $P^1$ and $C^1$. Since $\overline{\Lambda}_2$ is itself a two-copy Legendrian link each of those vector spaces decomposes further. We label each of the summand by the positive and negative ends of the chords. For instance $Q^1=Q^1_{11}\oplus Q^1_{22}\oplus Q^1_{12}\oplus Q^1_{21}\oplus Q^1_{f_1}$ where the differential splits as in the previous section. The vector space $Q^1\oplus C^1$ is a subcomplex for the linearised differential given by the augmentation $\varepsilon_{12}$ of $2\overline{\Lambda}_2$ given by $\varepsilon_1$ and $\varepsilon_2$. This subcomplex is shown to have the structure of a mapping cone of a map $\eta: (Q^1,d^{\varepsilon_{12}})\rightarrow (C^f,d^f)$. With respect to the decomposition of $Q^1$ this map turns out to be $\eta=\rho_{11}\oplus \rho_{22}\oplus 0 \oplus 0\oplus 0$; in other terms the vector space $Q^1_{12}\oplus Q^1_{21}\oplus Q^1_{f_1}$ is a subcomplex. 

The long exact sequence of equation \eqref{eq:10} is the one given by this mapping cone, this implies that the exact sequence for the link $\overline{\Lambda}_2$:
\begin{align}
  \cdots\rightarrow &H_{k+1}(\Lambda)\oplus H_{k+1}(\Lambda)\rightarrow\nonumber\\
 &LCH_{\varepsilon_1}^{n-k-1}(\Lambda)\oplus LCH_{\varepsilon_2}^{n-k-1}(\Lambda)\oplus LCH_{\varepsilon_2,\varepsilon_1}^{n-k-1}(\Lambda)\oplus
  H^{n-k-1}(C(\rho^*))\rightarrow \nonumber\\
&LCH^{\varepsilon_1}_{k}(\Lambda)\oplus LCH^{\varepsilon_2}_{k}(\Lambda)\oplus
  LCH_k^{\varepsilon_1,\varepsilon_2}(\Lambda)\oplus
  H_k(C(\rho))\rightarrow\cdots\label{eq:18}
  \end{align}

 splits as $4$ exact sequences, the first two are the exact sequences for $LCH^{\varepsilon_1}$ and $LCH^{\varepsilon_2}$ and the other two lead to isomorphisms:
$0\rightarrow LCH_{\varepsilon_2,\varepsilon_1}^{n-k-1}(\Lambda)\rightarrow H_kC(\rho) \rightarrow 0$ and $0\rightarrow H^{n-k-1}(C(\rho^*))\rightarrow LCH_k^{\varepsilon_1,\varepsilon_2}(\Lambda)\rightarrow 0$. This completes the proof.
\end{proof}

A difference with the duality exact sequence lies in the fact that since $\varepsilon_0$ is a priori different from $\varepsilon_1$ the manifold class (as in \cite{Duality_EkholmetAl}) is not well understood, this is due to the fact that one cannot describe the fundamental class (see Example \ref{sec:trefoil-knot}).

\section{Perspectives}
\label{sec:perspective}

\subsection{Lagrangian fillings}
In \cite{Ekholm_FloerlagCOnt}, Ekholm developed a Lagrangian intersection Floer theory for exact Lagrangian fillings related to the construction of the rational relative symplectic field theory of \cite{RSFT}. This construction is related to the wrapped Fukaya category of \cite{WrappedFuk}.  The convex part of this theory is the Lagrangian intersection theory of the Lagrangian fillings in the classical sense. 

We recall the definition of Lagrangian filling which is a special case of Lagrangian cobordism as in \cite{chantraine_conc}.

 \label{sec:geom-interpr-footn}
\begin{defn}
  \label{def:lagrangianfilling}
A Lagrangian filling of a Legendrian submanifold $\Lambda$ of a contact manifold $(Y,\xi)$ in $\mathbb{R}\times Y$ is a compact Lagrangian submanifold $L$ of $(-\infty,0)\times Y$ such that
\begin{enumerate}
\item[(i)] $\partial L=\Lambda$,
\item[(ii)] $L \cup (\mathbb{R_+}\times \Lambda)$ is a Lagrangian submanifold of $\mathbb{R}\times Y$.
\end{enumerate}
\end{defn}

In \cite{Kalman_&_Lagrangian_Cobordism} and \cite{Ekholm_FloerlagCOnt}, it is shown that an exact Lagrangian filling with vanishing Maslov class of a Legendrian knot leads to an augmentation of its Chekanov algebra. We denote the augmentation induced by $L$ by $\varepsilon_L$. A consequence of the results of \cite{Ekholm_FloerlagCOnt} is the following isomorphism:

$$LCH_{\varepsilon_L}^{n-k+2}(\Lambda)\simeq H_{k}(L).$$

The proof relies on the existence of an exact triangle involving $LCH^{\varepsilon_L}(\Lambda)$, $HF(L)$ and the full Lagrangian Floer homology of $L$. 

Similarly to the case of the wrapped Floer homology, in jet space the full Floer homology vanishes, so that $LCH^{\varepsilon_L}(\Lambda)$ and $HF(L)$  are isomorphic. Standard arguments then show that $HF(L)\simeq H(L)$. In the context of bilinearised Legendrian contact homology, only the latter isomorphism fails to be true as two different Lagrangian fillings might not be Hamiltonian isotopic. Namely from \cite[Theorem 4.9]{Ekholm_FloerlagCOnt} we have: 

\begin{Thm}
  \label{thm:Lagrangianfilling}
Let $L_1$ and $L_2$ be two exact Lagrangian filling of $\Lambda$ with vanishing Maslov class inducing augmentations $\varepsilon_{L_1}$ and $\varepsilon_{L_2}$ of $\mathcal{A}(\Lambda)$. Then
$$LCH^{n-k+2}_{\varepsilon_{L_1},\varepsilon_{L_2}}(\Lambda)\simeq HF^{k}(L_1,L_2).$$
\end{Thm}

The proof is exactly the same as in \cite[Section 4.4]{Ekholm_FloerlagCOnt} but we make no use of the conjectural Lemma 4.10 from there as we do not need to compare holomorphic curves with Morse gradient trajectories here.

\subsection{Relation with generating families}
\label{sec:relat-with-gener}

This project is part of an attempt in producing a unified picture of the Legendrian invariants arising from holomorphic curves on one side and generating families on the other. Recall that a generating family for a Legendrian submanifold $\Lambda$ in jet space $\mathcal{J}^1(M)$ is a function $F: M\times \mathbb{R}^m\rightarrow\mathbb{R}$ such that $\Lambda=\{(q_0,p_0,F(q_0,\eta_0))\vert \frac{\partial F}{\partial \eta}(q_0,\eta_0)=0\; p_0=\frac{\partial F}{\partial q}(q_0,\eta_0)\}$. To such a family one associates the difference function
\begin{equation*}
  \begin{array}{cccc}
\widetilde{F}: & M\times \mathbb{R}^m\times\mathbb{R}^m&\rightarrow&\mathbb{R}\\
&(q,\eta_1,\eta_2)&\mapsto&F(q,\eta_2)-F(q,\eta_1)
\end{array}
\end{equation*}
and define the \textit{generating family homology} to be $GF(\Lambda,F)=H_*(F\geq\epsilon,F=\epsilon)$. The set of those homology groups for all generating families form a Legendrian invariant bearing similarities with the linearised Legendrian contact homology. For instance, when $M=\mathbb{R}$, it has been shown in \cite{Fuchs_Rutherford_Generating_families} that $LCH^\varepsilon(\Lambda)\simeq GF(\Lambda,F^\varepsilon)$ where $F^\varepsilon$ is an equivalence class of generating family associated to $\varepsilon$ via graded ruling. It is a conjecture that this result holds in any dimension.

Bilinearised Legendrian contact homology appears naturally when trying to extend this conjectural isomorphism to a generalisation of generating family homology to an invariant communicated to us by Petya Pushkar and defined using the function:
\begin{equation*}
  \begin{array}{cccc}
\widetilde{F}: & M\times \mathbb{R}^m\times\mathbb{R}^m&\rightarrow&\mathbb{R}\\
&(q,\eta_1,\eta_2)&\mapsto&F_1(q,\eta_2)-F_2(q,\eta_1)
\end{array}
\end{equation*}

and constructing $GF(\Lambda;F_1,F_2)=H_*(\widetilde{F}\geq\epsilon,\widetilde{F}=\epsilon)$ where $F_1$ and $F_2$ are two generating families for $\Lambda$ (one assume that they have been stabilised to be defined on the same space). Conjecturally, those homologies correspond to the bilinearised Legendrian contact homology groups. More deeply, one can adapt a construction of Fukaya to construct an $\mathcal{A}_\infty$-category whose objects are generating families for $\Lambda$ and morphism spaces are the Morse complexes associated to $\widetilde{F}$. One expects the existence of an $\mathcal{A}_\infty$-functor from this category to $\mathit{Aug}(\Lambda)$. It is not expected that this functor is a quasi-equivalence, we however expect it to be cohomologically full and faithful and that it induces an equivalence of categories for the derived categories.

Note that the considerations in Section \ref{sec:geom-interpr-footn} also imply the existence of a cohomologically full and faithful functor from a Fukaya type category constructed using a Lagrangian filling of $\Lambda$ to $\mathit{Aug}(\Lambda)$ we similarely conjecture that this induce quasi-equivalence of the derived categories.

\section{Examples}\label{sec:exemple}
In this section we will show how bilinearised contact homology allows to distinguish augmentations of some DGAs. 

\paragraph{Trefoil knot.}
\label{sec:trefoil-knot}

We start with the example of the maximal Thurston-Bennequin right handed trefoil knot. It has $5$ augmentations all of which lead to isomorphic linearised contact homologies. However we will see that in the bilinearised contact homology table non diagonal terms are not isomorphic to any diagonal term. As a consequence, all those five augmentations are pairwise non homotopic. Moreover, all five augmentations arise from Lagrangian fillings of the trefoil knot. Thus the computation implies that those Lagrangian fillings are pairwise non symplectically equivalent (a result claimed by Ekholm, Honda and K\`alm\`an in \cite{Kalman_&_Lagrangian_Cobordism}).

\begin{figure}[ht!]
\labellist
\small\hair 2pt
\pinlabel {$b_1$} [bl] at 50 260
\pinlabel {$b_2$} [bl] at 230 260
\pinlabel {$b_3$} [bl] at 415 260
\pinlabel {$a_1$} [bl] at 575 350
\pinlabel {$a_2$} [bl] at 575 140
\endlabellist
  \centering
  \includegraphics[height=5cm]{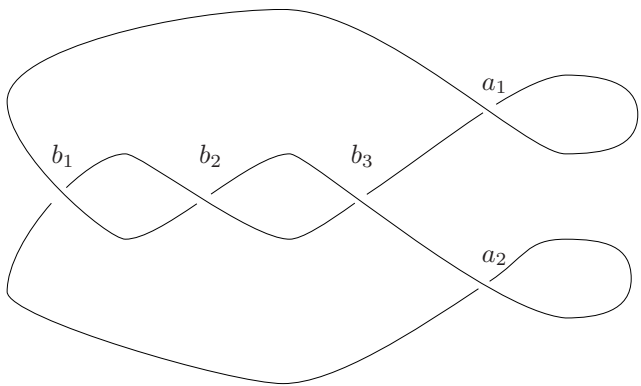}
  \caption{Right handed maximal $tb$ trefoil knot.}
  \label{fig:trefoil}
\end{figure}

The Chekanov algebra of $K$ has $5$ generators. Two of degree $1$: $a_1$ and $a_2$, and three of degree $0$: $b_1$, $b_2$ and $b_3$. The differential is given by:
\begin{align*}
\partial a_1=1+b_1+b_3+b_1b_2b_3,\\
\partial a_2=1+b_1+b_3+b_3b_2b_1.
\end{align*}

This DGA admits five augmentations listed in Table \ref{tab:Augment23}. 
\begin{table}
  \centering
  \begin{tabular}{c|ccc}
    & $b_1$ & $b_2$ & $b_3$ \\
    \hline
$\varepsilon_1$ &$1$ &$1$  &$1$ \\
$\varepsilon_2$ & $1$ &$0$ &$0$ \\
$\varepsilon_3$ & $1$&$1$ & $0$\\
$\varepsilon_4$ & $0$ & $0$&$1$ \\
$\varepsilon_5$ & $0$ &$1$ &$1$ 
  \end{tabular}
  \caption{Augmentations of the right handed maximal tb trefoil knot}
  \label{tab:Augment23}
\end{table}
 
An easy computation shows that $LCH^{\varepsilon_i,\varepsilon_j}(K)=\mathbb{Z}_2[0]$ as soon as $i\not= j$ and that $LCH^{\varepsilon_i,\varepsilon_i}(K)=\mathbb{Z}_2[1]\oplus \mathbb{Z}_2^2[0]$. These $5$ augmentations are thus all pairwise non-equivalent. As it is shown in \cite{Kalman_&_Lagrangian_Cobordism}, those $5$ augmentations arise in this context from $5$ different Lagrangian fillings of $K$ (one gets them by resolving the crossing $b_i$ in various orders). It follows from our computation and Theorem \ref{thm:Lagrangianfilling} that those Lagrangian fillings are non Hamiltonian isotopic. Note that it follows from \cite[Lemma 3.13]{Kalman_&_Lagrangian_Cobordism} that isotopic exact Lagrangian fillings lead to equivalent augmentations (the map $T^d$ is defined counting degree $-1$ holomorphic curves along a generic isotopy $\Sigma_t$).

\begin{figure}[ht!]
\labellist
\small\hair 2pt
\pinlabel {$m_1$} [br] at 40 430
\pinlabel {$m_2$} [br] at 40 205
\pinlabel {$b^1_{21}$} [br] at 70 315 
\pinlabel {$b^1_{11}$} [tl] at 100 290 
\pinlabel {$b^1_{12}$} [bl] at 130 345
\pinlabel{$b^1_{22}$} [bl] at  100 375
\pinlabel {$b^2_{21}$} [bl] at 230 330 
\pinlabel {$b^2_{11}$} [tl] at  270 285
\pinlabel {$b^2_{12}$} [bl] at 310 335
\pinlabel{$b^2_{22}$} [bl] at  275 360 
\pinlabel{$b^3_{21}$} [bl] at  425 330
\pinlabel {$b^3_{11}$} [tl] at 460 285
\pinlabel {$b^3_{22}$} [bl] at 460 365
\pinlabel {$b^3_{12}$} [bl] at 500 335
\pinlabel {$a^1_{21}$} [bl] at 535 420
\pinlabel {$a^1_{11}$} [bl] at 575 390
\pinlabel {$a^1_{12}$} [bl] at 660 440
\pinlabel {$a^1_{22}$} [bl] at 620 470
\pinlabel {$a^2_{21}$} [bl] at 535 205
\pinlabel {$a^2_{11}$} [bl] at 575 180
\pinlabel {$a^2_{12}$} [bl] at 655 240
\pinlabel {$a^2_{22}$} [bl] at 615 255
\pinlabel {$M_1$} [bl] at 775 445
\pinlabel {$M_2$} [bl] at 770 240
\endlabellist
  \centering
  \includegraphics[height=8cm]{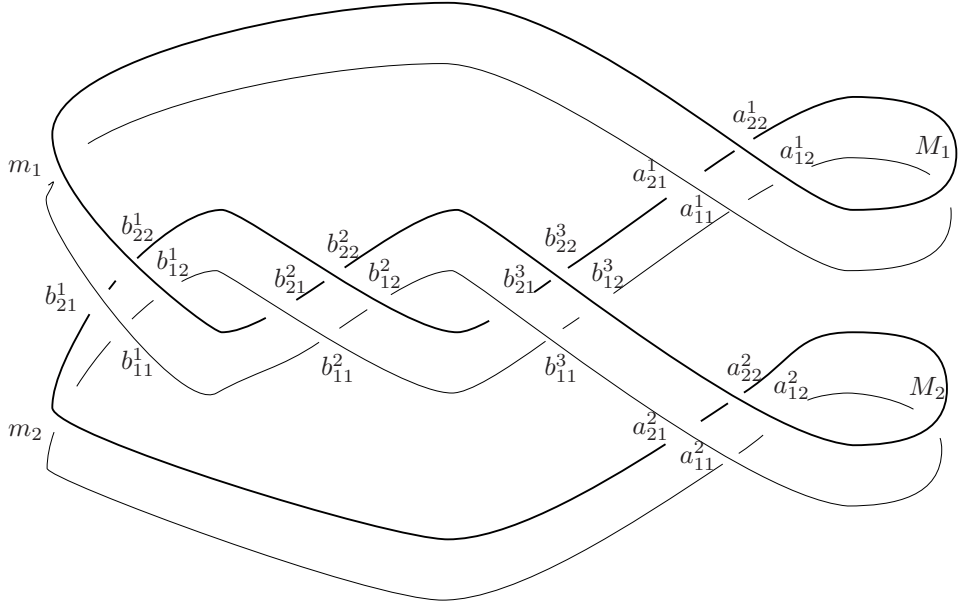}
  \caption{Two copies of the maximal $tb$ trefoil knot.}
  \label{fig:trefoil2copy}
\end{figure}

We compute also the maps of the exact sequence of Section \ref{sec:prod-struct-adn} in order to show that there is no fundamental class in the context of bilinearised Legendrian contact homology. To get the explicit maps one needs to consider the differential of the $2$-copy trefoil knot of Figure \ref{fig:trefoil2copy} (here the Morse function on $S^1$ has two minima $m_1$ and $m_2$ and two maxima $M_1$ and $M_2$). Specifically one needs to understand the map $\rho: C^{21}\rightarrow C^f$ of Section \ref{sec:structure-delta_n} arising while computing the differential $\partial$ on generators of type $\gamma_{12}$. From Figure \ref{fig:trefoil2copy} we get:
\begin{align*}
    \partial(a^1_{12})&=M_1+m_1a^1_{11}+(\cdots),\\
\partial(a^2_{12})&=M_2+a^2_{22}m_2+(\cdots),\\
\partial(b^1_{12})&=m_1b^1_{11}+b^1_{22}m_2,\\
\partial(b^2_{12})&=m_1b^2_{11}+b^2_{22}m_2,\\
\partial(b^3_{12})&=m_1b^3_{11}+b^3_{22}m_2.
\end{align*}

where the terms $(\cdots)$ do not involve any continuum chords and are characterised by Theorem~\ref{thm:equiv}.

This is enough to compute explicitly the long exact sequence. For $\varepsilon_1$ and $\varepsilon_2$ we get that
\begin{align*}
  \rho^{\varepsilon_1,\varepsilon_2}(a_1)&=0,\\
\rho^{\varepsilon_1,\varepsilon_2}(a_2)&=0,\\
\rho^{\varepsilon_1,\varepsilon_2}(b_1)&=m_1+m_2,\\
\rho^{\varepsilon_1,\varepsilon_2}(b_2)&=m_1,\\
\rho^{\varepsilon_1,\varepsilon_2}(b_3)&=m_1.
\end{align*}

In this situation, $LCH_0^{\varepsilon_1,\varepsilon_2}$ is generated by $[b_2]$. And the map $\rho^{\varepsilon_1,\varepsilon_2}$ is injective. Similar considerations imply that the map $\sigma:H_1(S^1)\rightarrow LCH^0_{\varepsilon_2,\varepsilon_1}$ is surjective. The map from $LCH^0_{\varepsilon_2,\varepsilon_1}\rightarrow LCH_0^{\varepsilon_1,\varepsilon_2}$ is thus $0$.

\paragraph{Chekanov-Eliashberg knot.}
Consider the Legendrian knot of Figure \ref{fig:chekeliash}. 

\label{sec:chek-eliashb-knot}
\begin{figure}[ht!]
\labellist
\small\hair 2pt
\pinlabel {$a_1$} [bl] at 34 155
\pinlabel {$a_2$} [bl] at 270 160
\pinlabel {$a_3$} [bl] at 50 31
\pinlabel {$a_4$} [bl] at 195 35
\pinlabel {$b_5$} [bl] at 157 171
\pinlabel {$b_6$} [bl] at 150 103
\pinlabel {$b_7$} [bl] at 95 80
\pinlabel {$b_8$} [bl] at 145 5
\pinlabel {$b_9$} [bl] at 225 70
\endlabellist
  \centering
  \includegraphics[height=8cm]{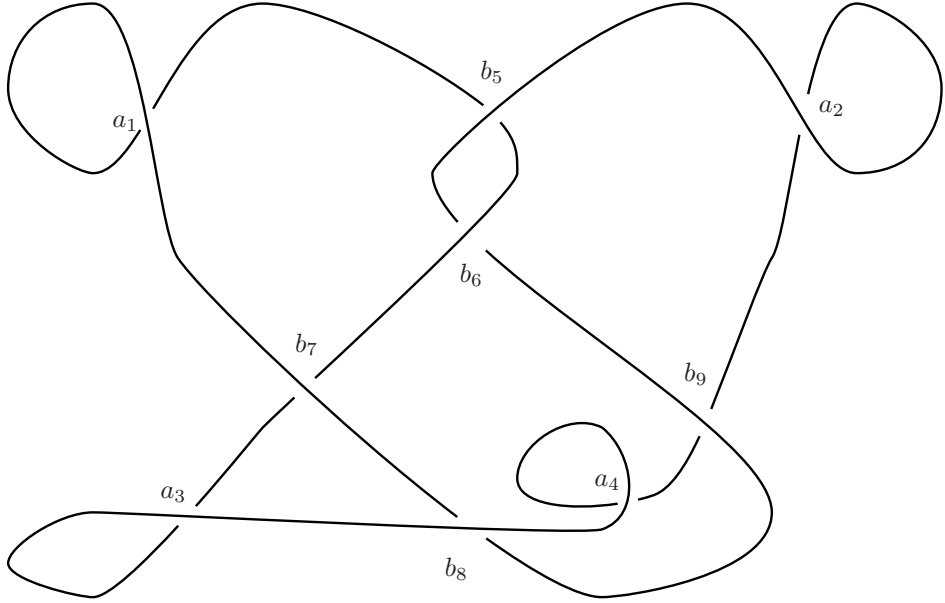}
  \caption{Chekanov-Elisahberg Legendrian knot.}
  \label{fig:chekeliash}
\end{figure}

It is one of the two Legendrian knots which have been distinguished in \cite{Chekanov_DGA_Legendrian} using Legendrian contact homology. The second one admits only one augmentation and hence has no non-trivial bilinearised Legendrian homology. However, this one admits three augmentations that are (as we will compute) non-equivalent. Again the degree of $a_i$ is $1$ and the degree of $b_i$ is $0$. The differential is given by
\begin{align*}
\partial(a_1)&=1+b_7+b_7b_6b_5+b_5+b_9b_8b_5,\\  
\partial(a_2)&=1+b_9+b_5b_6b_9,\\
\partial(a_3)&=1+b_8b_7,\\
\partial(a_4)&=1+b_8b_9.
\end{align*}
 And the augmentations are 

\begin{table}[h]
  \centering
  \begin{tabular}{c|ccccc}
    & $b_5$ & $b_6$ & $b_7$ &$b_8$ & $b_9$\\
    \hline
$\varepsilon_1$ &$0$ &$0$  &$1$ &$1$&$1$\\
$\varepsilon_2$ & $0$ &$1$ &$1$ &$1$&$1$ \\
$\varepsilon_3$ & $1$&$0$ & $1$ &$1$&$1$\\
  \end{tabular}
  \caption{Augmentations of the Chekanov-Eliashberg knot.}
  \label{tab:Augmentchekel}
\end{table}

For all of these augmentations, the linearised Legendrian contact homologies coincide 
and are given by $LCH^{\varepsilon_i}(K)\simeq \mathbb{Z}_2[1]\oplus\mathbb{Z}_2^2[0]$,
$i=1,2,3$.

However for any choice of a pair $\varepsilon_i\not=\varepsilon_j$ of augmentations, we get $LCH^{\varepsilon_i,\varepsilon_j}(K)\simeq \mathbb{Z}_2[0]$ and we deduce that these three augmentations are pairwise non-equivalent. 

As the second Chekanov-Eliashberg example has only one augmentation, we get that the set of equivalence classes of augmentations for those knots are different.

\bibliographystyle{plain}
 \bibliography{Bibliographie_en}

\textsc{D\'epartement de Math\'ematiques, Universit\'e Libre de Bruxelles, CP 218, Boulevard du Triomphe, B-1050 Bruxelles, Belgium}\\
\textit{Email address:}\textbf{fbourgeo@ulb.ac.be}

\textsc{Laboratoire de Math\'ematiques Jean Leray, BP 92208, 2 Rue de la Houssini\`ere, F-44322 Nantes Cedex 03, France}\\
\textit{Email address:}\textbf{baptiste.chantraine@univ-nantes.fr}

 \end{document}